\documentclass[12pt]{article}
\usepackage{amsfonts, amsmath,amsthm}
\usepackage{comment}
\usepackage{xcolor}

\numberwithin{equation}{section}
\renewcommand{\thesection}{\arabic{section}}
\renewcommand{\theequation}{\thesection.\arabic{equation}}

\newtheorem{theorem}[equation]{Theorem}
\newtheorem{lemma}[equation]{Lemma}
\newtheorem{corollary}[equation]{Corollary}
\newtheorem{proposition}[equation]{Proposition}
\newtheorem{definition}[equation]{Definition}

\newtheorem{remark}[equation]{Remark}
\newtheorem{example}[equation]{Example}
\newtheorem{question}[equation]{Question}

\numberwithin{equation}{section}
\renewcommand{\thesection}{\arabic{section}}
\renewcommand{\theequation}{\thesection.\arabic{equation}}

\newcounter{subeq}
\numberwithin{subeq}{equation}
\renewcommand{\thesubeq}{\theequation.\arabic{subeq}}

{\refstepcounter{subeq}\medskip\noindent(\thesubeq)\hfill}%
{\hfill\medskip}

\def\calC{\mathcal{C}}
\def\calO{\mathcal{O}}

\def\bbP{\mathbb{P}}

\def\bbX{\mathbb{X}}
\def\bbR{\mathbb{R}}

\newcommand{\vdim}{\operatorname{vdim}}

\def\mod{\text{mod}\;}

\begin{document}
\title{Coxeter theory for curves on blowups of $\bbP^r$}


\author{Olivia Dumitrescu and Rick Miranda}

\author{
	Olivia Dumitrescu  \\
	\small University of North Carolina at Chapel Hill \\ \small Chapel Hill, NC 27599-3250 \\
	\and
	Rick Miranda \\
	\small Colorado State University \\ \small Fort Collins, CO 80523 USA}

\maketitle

\begin{abstract}
We investigate the study of smooth irreducible rational curves
in $Y_s^r$, a general blowup of $\bbP^r$ at $s$ general points, whose normal bundle splits as a direct sum of line bundles all of degree $i$,
for $i \in \{-1,0,1\}$: we call these \emph{$(i)$-curves}.
We systematically exploit the theory of Coxeter groups
applied to the Chow space of curves in $Y_s^r$,
which provides us with a useful bilinear form that helps to expose properties of $(i)$-curves.
We are particularly interested in the orbits of lines (through $1-i$ points)
under the Weyl group of standard Cremona transformations (all of which are $(i)$-curves):
we call these \emph{$(i)$-Weyl lines}.
We prove various theorems related to understanding
when an $(i)$-curve is an $(i)$-Weyl line,
via numerical criteria expressed in terms of the bilinear form.
We obtain stronger results for $r=3$,
where we prove a Noether-type inequality that gives a sharp criterion.
\end{abstract}

\tableofcontents

\section{Introduction}
Let us consider the rational variety $Y=Y_s^r$
defined as the blowup of $\bbP^r$ at $s$ general points $p_1,\ldots,p_s$,
with blowup map $\pi:Y \to \bbP^r$.

In this paper we study properties of curves in the space $Y$.
The study of curves in projective space
is a well known problem in algebraic geometry,
that goes back centuries.
In recent work \cite{DM2}
the authors defined the concept of $(i)$-curves in $Y$
to be \emph{smooth, rational, irreducible curves}
with normal bundle splitting as a direct sum of
$\mathcal{O}_{\mathbb{P}^1} (i)$,
for $i\in \{-1, 0, 1\}$.
Prototypes for these curves are given by lines, passing through at most $2$ points.
For example, $(-1$)-curves are rigid,
and they give examples of flopping classes,
while $(0)$- and $(1)$-curves are movable
and they determine extremal rays
for the effective cone of divisors for $Y_s^r$.
The property of a space being a Mori Dream Space
can also be proved via the theory of movable curves,
(e.g. certain $Y_s^r$ spaces \cite{DM2}).
The minimal model program in birational geometry
has been formulated via the theory of divisors,
and it is an interesting question
to understand it via the theory of curves.

One can obtain $(i)$-curves in $Y_s^r$
by systematically applying standard Cremona transformations to lines through at most $2$ points.
Such an $(i)$-curve will be called an \emph{$(i)$-Weyl line},
and one of the goals of the paper is to understand when a given Chow class
can be an $(i)$-Weyl line.

In \cite{DM2} the authors proved that
when $Y$ is a Mori Dream Space,
the $(-1)$-Weyl curves can be identified numerically, using the parameters in the Chow class,
by a linear invariant known as the
\emph{anticanonical degree}
and a quadratic invariant,
the \emph{self-intersection} of the curve via a bilinear pairing
that emerges from the Coxeter theory of the Weyl group acting on the curve classes.
(This pairing is directly equivalent to the Dolgachev-Mukai pairing.)

However, for spaces that are not Mori Dream Spaces,
this result fails as Example \ref{examplesd=7and13} shows for $Y^3_{12}$.
In this paper we give a numerical characterization
for $(-1)$-Weyl curves in $\mathbb{P}^3$ 
by adding an additional numerical condition 
reflecting the irreducibility of the curve,
related to studying the projections of the curve.

In Section \ref{chow section}
we review the properties of the Chow ring and projections;
in Section \ref{cremona section}
we review the Cremona transformation action on curves,
and in Section \ref{weyl section}
we review the Coxeter theory and Weyl group action on the Chow curve classes $A^{r-1}(Y)$.
Our main results are the following:

\begin{itemize}
\item[(A)] The action of the Weyl group on the subspace of $A^{r-1}(Y)$ which is orthogonal to the canonical class is isomorphic to the standard geometric representation of the Coxeter group
associated to the graph $T_{2,r+1,s-r-1}$ (see Proposition \ref{actionW1V1}).
\item[(B)] For $r \geq 3$, all $(i)$-Weyl line classes of degree greater than one satisfy the inequality $d \leq \sum_{i\leq r+1} m_i$ when the multiplicities are put in non-increasing order; hence all such classes can be monotonically reduced to the line class through $1-i$ points by applying a sequence of standard Cremona transformations
(see Theorems \ref{thm:noCR-1}, \ref{thm:0weyllines}, \ref{thm:1weyllines}).
\item[(C)] We prove a direct analogue of Noether's Inequality for curves in $\bbP^3$,
indicating that under certain numerical conditions,
one can apply a standard Cremona transformation to a curve class to reduce the degree
(see Theorem \ref{noether}).
\item[(D)] We give a numerical criterion for a curve class to be an $(i)$-Weyl line in Theorem \ref{equivalence2}.
\end{itemize}



Theorem \ref{equivalence2} resembles the theory for divisors.
We leave as an open question if $(-1)$-Weyl cycles,
i.e. Weyl orbits of linear spaces,
can be equivalently defined by numerical conditions
induced by the Coxeter theory.
In this direction,
the paper \cite{DP} also gives equivalent conditions
for a divisor to be an $(i)$-Weyl divisor:
satisfying a linear and quadratic formula,
and with positive intersection with all $(-1)$-divisors
with respect to the Dolgachev-Mukai pairing. 

\subsection*{Acknowledgements} 
The collaboration  was partially supported by NSF grant DMS 1802082/ DMS 2041740 and by the Simons Foundation: Collaboration Grants 855897. The first author is a member of the Simion Stoilow Institute of Mathematics, Romanian Academy, 010702 Bucharest, Romania.

\section{The Chow ring of $Y_s^r$}\label{chow section}

\subsection{Generators and multiplication in $A(Y_s^r)$}

Let $H$ be the hyperplane class in $\bbP^r$,
and $E_i$, $1 \leq i \leq s$, be the exceptional divisors in $Y=Y_s^r$;
each $E_i$ is isomorphic to $\bbP^{r-1}$.

The Chow ring $A(Y) = \oplus_{0 \leq j \leq r} A^j(Y)$ of $Y$
is easy to describe.
Let $h_j = [\pi^*(H^j)]$ be the class of the pullback of the general linear space of codimension $j$.
Let $e_{i,j}$ be the class of the general linear subspace of $E_i$ of codimension $j$ in $Y$
(hence of codimension $j-1$ in $E_i$).
The following is standard; see for example \cite{EH}.

\begin{proposition}
With the above notation, we have:
\begin{itemize}
\item[(a)] The Chow ring $A(Y)$ is generated as a ring by the divisor classes
$h_1$, $e_{i,1}$, $1\leq i \leq s$.
\item[(b)] The Chow group $A^0(Y)$ is one-dimensional,
generated by the identity class $[Y]$;
dually, the Chow group $A^{r}(Y)$ is one-dimensional,
generated by the class $[p]$ of a point $p$
(which is equal to $h_r$ and $e_{i,r}$ for every $i$)
\item[(c)] The Chow groups $A^j(Y)$ for $1 \leq j \leq r-1$ have dimension $s+1$,
with basis
$h_j$, $e_{i,j}$, $1 \leq i \leq s$.
We set $h_j = e_{i,j} = 0$ for $j \geq r+1$.
\item[(d)] Multiplication in $A(Y)$ is induced from
the intersection form $(-\cdot -)$ on $Y$,
and is given by:
\[
(h_{j_1} \cdot h_{j_2}) = h_{j_1+j_2};\;\;
(h_j \cdot e_{i,\ell}) = 0;\;\;
(e_{i_1,j_1}\cdot e_{i_2,j_2}) = -\delta_{i_1i_2} e_{i_1,j_1+j_2};\;\;
h_r = e_{i,r} = [p].
\]
\item[(e)] The canonical class of $Y$ in $A^1(Y)$ is given by
\[
K_Y = -(r+1) h_1 + (r-1) \sum_{i=1}^s e_{i,1}
\]
and we define the \emph{anticanonical curve class} $F \in A^{r-1}(Y)$ to be
\[
F = (r+1)h_{r-1} - \sum_{i=1}^s e_{i,r-1}
\]
(which we will see is in some sense dual to $-K_Y \in A^1(Y)$).
\end{itemize}
\end{proposition}

Note that $[E_i] = e_{i,1}$ and $[E_i]^j = {(-1)}^{j-1} e_{i,j} \in A^j(Y)$.

For convenience we abbreviate the divisor classes as $H = h_1$ and $E_i = e_{i,1}$ in $A^1(Y)$;
similarly if there is no possibility of confusion we will abbreviate the curve classes
as $h = h_{r-1}$ and $e_i = e_{i,r-1}$ in $A^{r-1}(Y)$.

If $a$ and $b$ are two classes in complementary dimension,
then we will usually abbreviate their product in $A(Y)$ as $(a\cdot b) = k$
when $(a \cdot b) = k [p]$.

If $\bar{C} \subset \bbP^r$ is a curve of degree $d$
having multiplicity $m_i$ at $p_i$ for each $i$,
and $C \subset Y$ is the proper transform in $Y$,
then the class $[C] \in A^{r-1}(Y)$ is
$
[C] = d h - \sum_i m_i e_{i}
$
which can be easily deduced by intersecting with the basis elements of $A^1(Y)$.
Similarly for divisors, if $D$ is a divisor on $\bbP^r$
having multiplicity $m_i$ at $p_i$ for each $i$,
then the proper transform of $D$ in $Y$ has the class $d H - \sum_i m_i E_i$.

The degrees and multiplicities as recorded in the respective Chow classes
are called the \emph{numerical characters} of the divisor or curve.

\subsection{Rational curves and $(i)$-curves ($i \in \{-1,0,1\}$)}
\label{rationalcurves(i)curves}
Suppose $C$ is a smooth rational curve in $Y=Y^r_s$,
whose class $[C] \in A^{r-1}(Y)$ is $[C]=d h - \sum_i m_i e_i$.
We define the \emph{virtual dimension} of $C$
in terms of these numerical characters to be
\begin{equation}\label{param}
\vdim(C):=(r+1)(d+1) - (r-1)\sum_{i=1}^s m_i - 4 = (r-3)-(C\cdot K_Y)
\end{equation}
which is $\chi$ of the normal bundle of $C$.
(It is also seen to be the dimension count of the number of parameters for a map from $\bbP^1$ to $\bbP^r$ of degree $d$, imposing the multiple points,
and taking into account the automorphisms of $\bbP^1$.)
It is determined by the class of $C$ in $A^{r-1}$.

If we assume that the number of parameters for such a curve $C$ is non-negative
(so that in general we can hope that
such a rational curve $C$ with these multiplicities is expected to exist)
but that it is isolated and is a reduced point in the Hilbert scheme
(it doesn't move in a family, even infinitesimally),
then we are imposing that $H^0 = H^1 = 0$ for the normal bundle,
which is equivalent to having the normal bundle  $N = N_{C/Y}$
split as a direct sum of $\calO(-1)$'s.

If $N$ splits as a direct sum of $\calO(i)$'s for $i \in \{-1,0,1\}$,
then the normal bundle has $H^1=0$ and $\dim H^0 = (r-1)(i+1)$,
so that $\chi(N) = (r-1)(i+1)$
and $(C\cdot K_Y) = (r-3)-(r-1)(i+1) = -2-i(r-1)$.
We call such a curve an \emph{$(i)$-curve}.

Motivated by this, we define a \emph{numerical $(i)$-class} in $A^{r-1}(Y)$
to be a class $c$
such that $(c\cdot K_Y) = -2-i(r-1)$;
this is the $\chi(N)=(r-1)(i+1)$ condition.
We can also write this as
$(r+1)d - (r-1)\sum_i m_i = 2+i(r-1)$.
The class of every $(i)$-curve is a numerical $(i)$-class of course;
the converse is not true in general.

We expect a finite number of smooth rational $(-1)$-curves
representing a given numerical $(-1)$-class.
The prototype of a $(-1)$-curve on $Y$
is the proper transform of a line through two of the points:
such a curve will have the class $h - e_i - e_j$
for two indices $i,j$.

The prototype of a $(0)$-curve is the proper transform of a line through one of the points;
such a curve will have class $h-e_i$ for some index $i$.
The prototype of a $(1)$-curve is a general line in $Y$, with class $h$.


\subsection{Projections}\label{projections}
Projections offer an interesting perspective and tool
to study curves in $Y=Y^r_s$.
If $C$ is an irreducible curve in $\bbP^r$, $r \geq 3$,
with class $c = dh - \sum_i m_i e_i \in A^{r-1}(Y^r_s)$,
then we may consider the projection $\pi$ from any one of the multiple points
(say $p_1$ with multiplicity $m_1$)
and obtain the curve $\pi(C) \subset \bbP^{r-1}$
with class 
$\pi(c) = (d-m_1) h - \sum_{j \geq 2} m_j e_j \in A^{r-2}(Y^{r-1}_{s-1})$.

We've seen above that $\chi(N_{C/Y})$ is determined by the class of $C$,
and it can happen that $\chi(N)$ is non-negative for $C$ but negative for $\pi(C)$.
This is an indication that,
although the $\chi(N)$ computation suggests that a rational curve exists with that class,
the projection is not expected to exist
(and hence the original curve actually shouldn't either).

We can say more in case $C$ is an $(i)$-curve in $\bbP^r$.
In that case, since the normal bundle $N_C$ 
will surject onto the normal bundle for $N_{\pi(C)}$,
we must have that that all of the summands for $N_{\pi(C)}$ have degree at least $i$.
Therefore $\chi(N_{\pi(C)}) \geq (r-2)(i+1)$.

One can then compute that
$\vdim(C)-\vdim(\pi(C)) = d+m_1 + 1 - \sum_{j=2}^s m_j $,
and since this difference is the difference of the $\chi$'s of the normal bundles,
we have that
\[
d+m_1 + 1 -  \sum_{j=2}^s m_j = (r-1)(i+1)-\chi(N_{\pi(C)}) 
\leq (r-1)(i+1)-(r-2)(i+1) = i+1.
\]
This leads to the following,
which we call the {\bf Projection inequality for $(i)$-curves}:

\begin{equation}\label{projectioninequalityicurves}
d+m_1 \leq \sum_{j=2}^s m_j + i
\end{equation}

A bit of algebraic manipulation shows that
if $c = dh-\sum_j m_j e_j$ is a numerical $(i)$-class,
so that $(r+1)d - (r-1)\sum_j m_j = 2+i(r+1)$,
then this projection inequality is equivalent to $d-(r-1)m_1 \geq 1$,
which applied to each multiplicity would imply
\begin{equation}\label{projectioninequalityiclass}
d-(r-1)m_j \geq 1  \;\;\text{or}\;\; m_j \leq \frac{d-1}{r-1}\;\;\text{ for each } j = 1,\ldots s;
\end{equation}
we will see that this is a useful formulation.

We developed this with $r \geq 3$ primarily in mind;
however it is instructive to see what the projection inequality for $(i)$-curves
implies when $r=2$.
For planar $(i)$-curves (studied previously in \cite{DO}),
the condition imposed by the projection inequality is rather mild:
it simply says that $d > m_j$ for each $j$.
Being an $(i)$-curve in the plane corresponds to having the anticanonical degree
(the intersection with $-K$)
equal to $A=i+2$; hence for $i \in \{-1,0,1\}$,
these are curves with anticanonical degree equal to $A=1,2,3$.
If we allow $d=0$, then the only possibility is an exceptional curve $e_j$ itself.
We have the following.

\begin{lemma}\label{planarcase}
For $r=2$ and $i \in \{-1,0,1\}$,
the only $(i)$-curves in $Y^2_s$ that fail the projection inequality
are lines through $1$ or $2$ points, or the exceptional divisors.
There are no curves with negative anticanonical degree that fail the projection inequality.
\end{lemma}

\begin{proof}
We have discussed the $d=0$ case, leading to the exceptional divisors,
so let us assume $d \geq 1$.
If the projection inequality fails, then one of the multiplicities (say $m_1$)
is at least $d$; and therefore must be equal to $d$.
Hence the $(i)$-curve is a cone, and since it is irreducible it must be a line.
This proves the first statement.

Suppose an effective divisor $D \equiv dh-\sum_j m_je_j$ fails the projection inequality;
then as above the divisor is a cone (over the first point, say),
and if $L_j$ is the line joining $p_1$ to $p_j$ for $j \geq 2$,
we must have $D = \sum_{j \geq 2} m_j L_j$, and $d = m_1 = \sum_{j \geq 2} m_j$.
Therefore $A = 3d-\sum_j m_j$ must equal $d$,
and cannot be negative.

\end{proof}

\section{The standard Cremona transformations}\label{cremona section}

\subsection{The geometric description of the standard Cremona transformation}
We will exploit the standard Cremona transformation of $\bbP^r$
(centered at $r+1$ points) which we describe below.
The reader may also see \cite{DM2} for a bit more detail.

The standard Cremona transformation of $\bbP^r$
(inverting the coordinates,
i.e., sending $[x_0:\cdots:x_r]$ to $[x_0^{-1}:\cdots:x_r^{-1}]$)
is realized geometrically by blowing up the $r+1$ coordinate points,
then all coordinate lines, then all coordinate $2$-planes, etc., until one blows up all coordinate $(r-2)$-planes;
then one blows down the exceptional divisors starting with those over the coordinate lines,
then those over the coordinate $2$-planes, etc.,
finally ending by contracting the proper transforms of the coordinate hyperplanes.

After the first stage of performing the blowups, we arrive at a $r$-fold $\bbX^r$;
this is the space that resolves the indeterminacies of the Cremona on $\bbP^r$.
The subspace in the codimension one part of the Chow ring of $\bbX^r$
generated by the exceptional divisors over the positive-dimensional linear coordinate spaces of $\bbP^r$
is invariant under the Cremona transformation,
and the quotient space inherits the action.
This quotient space is naturally isomorphic to $A^1(Y^r_{r+1})$,
and the action extends (trivially) to an action on $A^1(Y^r_s)$ for any $s \geq r+1$.

Since the $s$ points are general, any set of $r+1$ of them
can be the base points of a corresponding Cremona transformation.
For any subset $I$ of $r+1$ indices,
we will denote by $\phi_I$ the corresponding Cremona transformation,
which induces an action on the codimension one Chow space $A^1(Y^r_s)$.

Dually, the subspace of the curve classes in the Chow ring of $\bbX^r$
spanned by the general line class and the general line classes inside of each $E_i$
is also invariant under the Cremona action.
This subspace is naturally isomorphic to $A^{r-1}(Y^r_{r+1})$,
and therefore we have a Cremona action there, which extends to $A^{r-1}(Y^r_s)$.

\subsection{Formulas for the Cremona action on divisor and curve classes}
We describe these actions explicitly below, and leave the details of checking the formulas to the reader.

\begin{proposition}\label{CremonaAY}
Fix any $(r+1)$-subset $I \subset \{1,2,\ldots, s\}$.
\begin{itemize}
\item[(a)]
The action of $\phi_I$ on $A^1(Y^r_s)$ is given by sending $dH-\sum_i m_i E_i$ to $d'H-\sum_i m_i' E_i$ where
\begin{align*}
d' &= rd - \sum_{j\in I}m_j = d+t_1 \\
m_i' &= (r-1)d - \sum_{j\in I, j\neq i} m_j = m_i + t_1 \;\;\text{for}\;\; i \in I \;\;\text{and}\;\;
m_i' = m_i \;\;\text{for}\;\; i \notin I
\end{align*}
and $t_1 = (r-1)d - \sum_{i \in I} m_i$.  It has order two.
In particular
\[
\phi_I(H) = rH-(r-1)\sum_{i\in I} E_i, \;\;\;
\phi_I(E_j) = H-\sum_{i\in I, i\neq j} E_j \text{ if } j\in I,
\]
and $\phi_I(E_j) = E_j$ if $j \notin I$.

\item[(b)]
The action of $\phi_I$ on $A^{r-1}(Y^r_s)$ is given by sending $dh-\sum_i m_i e_i$ to $d'h-\sum_i m_i' e_i$ where
\begin{align*}
d' &= rd-(r-1)\sum_{i\in I} m_i =d+(r-1)t_{r-1} \\
m_i' &= d - \sum_{j\in I, j \neq i} m_j = m_i + t_{r-1} \;\;\text{for}\;\; i \in I \;\;\text{and}\;\;
m_i' = m_i \;\;\text{for}\;\; i \notin I
\end{align*}
and $t_{r-1} = d - \sum_{i \in I} m_i$.  It has order two.
In particular
\[
\phi_I(h) = rh-\sum_{i\in I} e_i,\;\;\;
\phi_I(e_j) = (r-1)h-\sum_{i\in I, i\neq j} e_j \text{ if } j\in I,
\]
and $\phi_I(e_j) = e_j$ if $j \notin I$.

\item[(c)]
The intersection pairing between $A^1(Y)$ and $A^{r-1}(Y)$ is $\phi_I$-invariant, i.e.,
for any class $D \in A^1(Y)$ and $C \in A^{r-1}(Y)$, we have
\[
(D \cdot C) = (\phi_I(D) \cdot \phi_I(C)).
\]
\end{itemize}
\end{proposition}

We will abbreviate $\phi = \phi_I$ for $I = \{1,2,\ldots,r+1\}$.

We will call a divisor class $dH-\sum_i m_i E_i$ \emph{Cremona reduced}
if $(r-1)d \geq \sum_{i \in I} m_i$ for every $(r+1)$-subset $I$ of the indices;
by (a) above, this implies that the degree of the class cannot be lowered
by any of the $\phi_I$'s.

Similarly, we call a curve class $dh-\sum_i m_i e_i$ \emph{Cremona reduced}
if $d \geq \sum_{i \in I} m_i$ for every $(r+1)$-subset $I$ of the indices;
by (b) above, this implies that the degree of the class cannot be lowered
by any of the $\phi_I$'s.

Certain other inequalities are common
when considering the degrees and multiplicities for classes of curves in $\bbP^r$.

\begin{lemma}\label{irrednondeg}
Suppose that $s \geq r+1$,
so that the points span $\bbP^r$.
Let $C$ be an irreducible non-degenerate curve in $\bbP^r$,
such that the class of $C$ in $A^{r-1}$ is $dh-\sum_{i=1}^s m_i$.
Then for any subset $J \subset \{1,\ldots,s\}$, if $1 \leq |J| \leq r$, we have
$d \geq \sum_{i \in J} m_i$.
\end{lemma}

\begin{proof}
If $d < \sum_{i \in J} m_i$, then since $C$ is irreducible
$C$ must be contained in any hyperplane that contains the points of $J$.
Therefore $C$ is contained in the intersection of all such,
which is the span of $J$.
Since $|J| < r+1$, the span of $J$ is a proper linear space,
so $C$ must be degenerate.
\end{proof}

This statement is familiar to those who know that the smallest degree of an irreducible nondegenerate curve in $\bbP^r$ is $r$, and that is achieved only by the rational normal curve; any such can be specified by $s=r+3$ points in general position.
(Here $d=r$ and all $m_i=1$.)

The above Lemma allows us to easily conclude the following.

\begin{proposition}
Suppose that $s \geq r+1$,
so that the points span $\bbP^r$,
and let $I$ be a subset of points with $|I|=r+1$.
If $C$ is an irreducible non-degenerate curve in $\bbP^r$ of degree $d > 1$,
then its image $\phi_I(C)$ after a Cremona transformation based at the points of $I$
is an irreducible curve with non-negative degree and multiplicities.
\end{proposition}

\begin{proof}
Since $C$ is non-degenerate and irreducible, it does not live in any hyperplane,
and so is not contracted by the Cremona transformation.
Hence $\phi_I(C)$ is an irreducible curve.
Let us denote its degree and multiplicities by $d'$ and $m_i'$.
The formulas of Proposition \ref{CremonaAY}(b)
show immediately that all the $m_i'$ are non-negative, using Lemma \ref{irrednondeg}.

Of course the degree is non-negative; to see this using the formulas,
note that if we add up the formulas $d \geq \sum_{i\in J} m_i$
for each $r$-subset $J \subset I$, we obtain
$r\sum_{i\in I} m_i \leq (r+1) d$, and so
\[
\sum_{i\in I} m_i \leq \frac{r+1}{r} d < \frac{r}{r-1} d
\]
which is equivalent to $d' > 0$ using the formula for $d'$ in Proposition \ref{CremonaAY}(b).
\end{proof}

\subsection{Cremona Invariants and $(i)$-Weyl lines}\label{bilinear form section}
Note that the symmetric group on the indices of the $s$ points
acts on all these spaces,
and if $\sigma$ is the permutation taking the subset $I$ to the subset $J$,
then $\sigma \phi_I \sigma^{-1} = \phi_J$.
The canonical class $K_Y = -(r+1)H + (r-1)\sum_i E_i$ is the only symmetric $\phi$-invariant class
(up to scalars) in $A^1$ (and it is invariant under all $\phi_I$ therefore).
Dually, the anticanonical curve class $F =(r+1)h - \sum_{i=1}^s e_i $
is the only symmetric $\phi$-invariant class (up to scalars) in $A^{r-1}(Y)$,
and again it is invariant under all $\phi_I$.

The invariance of these classes gives us linear functionals
\[
(-,F)\in (A^1)^* \;\;\text{and}\;\; (-,K) \in (A^{r-1})^*
\]
that are both symmetric and Cremona-invariant.

In addition to these familiar linear functionals, both spaces enjoy a symmetric $\phi$-invariant quadratic form,
which are unique up to scalars also.
These we denote by $q_1$ on $A^1$ and $q_{r-1}$ on $A^{r-1}$, and may be defined using the numerical characters
$(d;\underline{m})$ by
\begin{align}
q_1(d;\underline{m}) &= q_1(dH - \sum_i m_i  E_i) = (r-1)d^2 - \sum_{i=1}^s m_i^2 \label{q1}\\
q_{r-1}(d;\underline{m}) &= q_{r-1}(dh - \sum_i m_i e_i) =d^2 - (r-1)\sum_{i=1}^s m_i^2 \label{qr-1}
\end{align}
Each of these quadratic forms give rise to associated bilinear forms,
which in the case of divisors is (up to scalar multiple) the Dolgachev-Mukai pairing.

We will denote the  bilinear form on $A^{r-1}$ 
derived from the quadratic form 
$q_{r-1}(dh-\sum_i m_i e_i) = d^2-(r-1)\sum_i m_i^2$
by $\langle -,-\rangle$:
$
\langle x, y \rangle = \frac{1}{2}(q_{r-1}(x+y)-q_{r-1}(x)-q_{r-1}(y))
$
which has the following values on the basis elements:
\[
\langle h,h\rangle = 1; \;\;
\langle h,e_i\rangle = 0 
\;\text{and}\;\langle e_i,e_i \rangle = 1-r \;\text{for all}\;i;
\;\; \langle e_i,e_j\rangle =0 \;\text{for}\; i \neq j
\]
It is easy to see that for any $c \in A^{r-1}$, we have
the \emph{anticanonical degree} of $c$ which is equal to
$(-K\cdot c) = \langle F , c\rangle$.
We therefore have that a class $c$ is a numerical $(-1)$-class
if and only if $\langle F, c\rangle = 3-r$
(and in general is a numerical $(i)$-class if and only if
$\langle F, c\rangle = 2+i(r-1)$)

For $i \in \{-1,0,1\}$,
we define an \emph{$(i)$-Weyl line}
to be the image
(under a sequence of standard Cremona transformations)
of a (degree one) line through $1-i$ points.
The class of a $(-1)$-Weyl line is therefore an element of the orbit
(under the group of Cremona transformations)
of the class $h-e_1-e_2$;
similarly the class of a $(0)$-Weyl line is an element of the orbit of $h-e_1$,
and the class of  $(1$)-Weyl line is an element of the orbit of the line class $h$.

Since the linear functional $\langle F, - \rangle$
and the quadratic form $\langle -,-\rangle$
are Cremona-invariant,
any $(i)$-Weyl line must have the same values for these invariants
as the lines through $1-i$ points.
Therefore:
\begin{corollary}\label{iWeyllineinvariants}
If $c$ is the class of an $(i)$-Weyl line in $\bbP^r$, we have
\[
\langle F, c \rangle = 2+i(r-1) \text{ and } \langle c, c \rangle = 1 + (i-1)(r-1).
\]
\end{corollary}

\subsection{Projections and Cremona transformations}\label{projectionsCremona}

It is easy to see that projections commute with the Cremona transformation, 
in the following sense (see \cite{DM2}).
Suppose that $I$ is any subset of $r+1$ indices for the $s$ points of $Y^r_s$,
and $i \in I$. 
If we denote the projection from the $i$-th point by $\pi_i$,
then we have
\[
\pi_i(\phi_I(c)) = \phi_{I-\{i\}}(\pi_i(c))
\]
which is elementary and we leave to the reader to check.

It is also the case that the projection inequality (\ref{projectioninequalityicurves})
is preserved under 'increasing' Cremona transformations, in the following sense.

\begin{proposition}\label{cremonaprojection}
Fix $i$, and let $c=dh-\sum m_ke_k$ be a curve class satisfying the inequality  \eqref{projectioninequalityicurves}:
\[
d+m_j \leq \sum_{k\neq j} m_k + i \text{ for every } j \in \{1,\ldots, s\}.
\]
Fix a subset $I\subset \{1,\ldots,s\}$ with $|I|=r+1$,
and write the Cremona transformed class $\phi_I(c)$ as $d'h-\sum_i m_k'e_k$.
If $d' \geq d$, then the class $\phi_I(c)$ also satisfies the inequalities
\[
d'+m_j' \leq \sum_{k\neq j} m_k' + i \text{ for every } j\in \{1,\ldots, s\}
\]
\end{proposition}

\begin{proof}
From Proposition \ref{CremonaAY}(b) if we set $t=d-\sum_{\ell \in I}m_\ell$,
we have $d'=d+(r-1)t$, 
and for any $k\in I$ we have $m'_k=m_k+t$;
for $k \notin I$, $m_k'=m_k$.
Our assumptions imply that $t \geq 0$.
We distinguish two cases.
If $j\in I$, then both sides of the projection inequality
change by adding $rt$, and hence it is preserved.

If $j \notin I$, then the left-hand side of the projection inequality changes by adding $(r-1)t$,
while the right-hand side changes by adding $(r+1)t$.
Since $t \geq 0$, the inequality is again preserved.
	
\end{proof}

We note that the projection inequality for numerical $(i)$-classes
\eqref{projectioninequalityiclass}
can be re-written in terms of the bilinear form, as follows:
\begin{equation}\label{bilinearprojinequality}
\langle C,h-e_j\rangle \geq 1 \;\;\text{ for each }\;\; j = 1,\ldots,s.
\end{equation}

\subsection{Irreducible and Non-degenerate curves}

Now let $C$ be an irreducible non-degenerate $(-1)$-curve in $\bbP^r$
whose class in $A^{r-1}$ is $c=dh-\sum_{i=1}^s m_i$,
and order the points so that $m_i \geq m_{i+1}$ for all $i < s$.

Let us also assume that $C$ is in the orbit of 
the line through two points under the Weyl group.
In that case there is a Cremona transformation taking the line through two points to $C$,
and that Cremona transformation $w$ will take the class $h-e_1-e_2$
to the class $c$ of $C$.
Conversely, if $c$ is a class in the orbit of $h-e_1-e_2$,
then there is an element $w$ in the Weyl group
taking the line through two points to a curve $C$ whose class is $c$.
(Just take $w$ to be the Weyl group element that takes $h-e_1-e_2$ to $c$.)
The prototype of such a class is the class of the rational normal curve of degree $r$
passing through $r+3$ general points of $\bbP^r$;
its class can be written as $(r;1^{r+3},0^{s-r-3})$.

\begin{lemma}\label{numericallemma}
Let $c$ be a class in $A^{r-1}(Y^r_s)$ with positive degree $d$ and non-negative multiplicities in non-decreasing order, such that $d \geq m_1$.
Assume that $c$ has the two numerical invariants of a $(-1)$-Weyl line class:
$\langle F,c\rangle = 3-r$ and $\langle c,c\rangle = 3-2r$.
\begin{itemize}
\item[(a)] If $d=1$, then $c = h-e_1-e_2$:
$c$ is the class of a line through two points.
\item[(b)] If $d > 1$, then $d \geq r$.
Moreover if $d = r$ the $c = dh-\sum_{1\leq j\leq r+3} e_j$:
$c$ is the class of a rational normal curve of degree $r$ through $r+3$ general points.
\item[(c)] If $d \geq r$
then $m_{r+3}\neq 0$:
there are at least $r+3$ strictly positive multiplicities.
\item[(d)] Suppose that $c$ is actually a $(-1)$-Weyl line class,
i.e., it is in the orbit of $h-e_1-e_2$.
Then $d \geq \sum_{i\leq r} m_i$.
\item[(e)] If $c$ is a $(-1)$-Weyl line class
and $c$ is not Cremona reduced (i.e., $d < \sum_{i=1}^{r+1} m_i$),
then after applying the standard Cremona transformation $\phi$ 
to the $r+1$ largest multiplicities, 
we obtain a class with positive degree and non-negative multiplicities.
\end{itemize}
\end{lemma}

\begin{proof}

To see (a),
if $d=1$, $\langle F,c\rangle=3-r$ says $(r-1)\sum_i m_i = 2r-2$,
so that the sum of the multiplicities is two. 
Since the multiplicities are at most the degree,
we must have two of them equal to one and all other equal to zero.

Statement (b) follows from re-writing the two invariants, as
\[
(r-1)\sum_i m_i = d(r+1)-(3-r) \;\text{ and }\;
(r-1)\sum_i m_i^2 = d^2-(3-2r).
\]
Since all multiplicities are non-negative, we have that their sum 
is not larger than the sum of their squares.  Hence we conclude that
$
d(r+1)-(3-r) =dr+d+r-3 \leq d^2-3+2r
$
which we can re-write as $(d-1)r \leq d^2-d=d(d-1)$,
from which we conclude that $r \leq d$.

If $d=r$, we see that in the argument above
the inequalities above must all be equalities,
and in particular $\sum m_i = \sum m_i^2$, which forces all nonzero $m_i$
to be equal to one, and gives the rational normal curve class.

To show (c), assume that there are exactly $k$ strictly positive multiplicities.
In that case we can apply the Cauchy-Schwarz inequality to conclude that
$(\sum m_i)^2 \leq k\sum m_i^2$.
We then see that
\[
(\frac{rd+d+r-3}{r-1})^2 = (\sum m_i)^2 \leq k\sum m_i^2 = k\frac{d^2+2r-3}{r-1}
\]
using the two invariants.
We want to prove that $k \ge r+3$, so it suffices therefore to show that
\[
r+3 \leq \frac{(rd+d+r-3)^2}{(r-1)(d^2+2r-3)}
\]
The denominators here are all positive, so clearing them,
expanding, simplifying, and factoring, this inequality is equivalent to
\[
(d-r)(4d+2r^2-6) \geq 0
\]
which is true since $d \geq r$, and proves (c).


If $c$ is the class of a $(-1)$-Weyl line $C$,
then by (c), we have that $C$ is non-degenerate;
it passes through at least $r+1$ points in general position.
Since $C$ is irreducible, Lemma \ref{irrednondeg} implies (d).

If in addition $c$ is not Cremona reduced,
then applying $\phi$ to $c$,
and setting $t = d-\sum_{j\leq r+1} m_j$, we obtain $c'=d'h-\sum m_i'$
where for $i \leq r+1$, $m_i' = m_i+t = d-\sum_{j \leq r+1, j\neq i} m_j$
using the formulas of Proposition \ref{CremonaAY}(b).
This quantity is non-negative by part (d).
Hence all multiplicities for $c'$ are non-negative,
and since $c'$ also represents an irreducible curve (the image of the curve $C$),
it must have positive degree also.
(Positive degree also follows easily now from observing that
$d' = rd-(r-1)\sum_{i\leq r+1} m_i
= (r-1)(d-\sum_{i\leq r} m_i) + (d-(r-1)m_{r+1})$;
the first quantity is non-negative 
and the second is strictly positive using Lemma \ref{irrednondeg}.)
This proves (e).

\end{proof}

Part (e) of Lemma \ref{numericallemma} will allow us to apply an inductive argument
to reduce $(-1)$-Weyl line classes via standard Cremona transformations,
as long as we can assure that the classes are not Cremona reduced.
For this we use a more sophisticated analysis,
relying on Coxeter systems and the properties of the Weyl group.

\section{The Weyl Group}\label{weyl section}

\subsection{Coxeter Systems and Groups of type $T_{2pq}$}
Fix $a$, $b$, and $c$ at least $2$.
Define the graph $T_{abc}$ with $a+b+c-2$ vertices,
and consisting of three chains
(of lengths $a$, $b$, and $c$) with the final vertices identified.
Hence there is one vertex of degree $3$ (the identified final vertex),
three vertices of degree $1$ (the other ends of the three chains)
and all other vertices have degree $2$.

For our purposes in studying the space $Y^r_s$,
we will have need of the case $a=2$, $b=r+1$, and $c = s-r-1$.
The total number of vertices is $s$.
We label them as follows, from $0$ to $s-1$:
\[
\begin{matrix}
(1)  -  (2)  -  \cdots  -  (r)  - &  (r+1) & -  (r+2)  \cdots  -  (s-1) \\
     &     |           &    \\
     &     (0)       &     \\
\end{matrix}
\]
For the basics of the Coxeter theory that applies to this situation,
we refer the reader to \cite{Bo} and \cite{Humphreys}.
The graph $T_{2,r+1,s-r-1}$ defines the Coxeter system
and the Coxeter group $W$
generated by $w_i$, $0 \leq i \leq s-1$,
with relations
$w_i^2=1$,
$w_iw_j = w_jw_i$ if $i$ and $j$ are not connected with an edge in the $T_{2,r+1,s-r-1}$ graph,
and $w_iw_jw_i = w_jw_iw_j$ if $i$ and $j$ are connected with an edge.

Let $E$ be the real vector space with basis $\{x_i\;|\; 0\leq i \leq s-1\}$.
Define the symmetric bilinear form $B$ on $E$ by
\[
B(x_i,x_j) = \begin{cases}
1 & \text{if}\;\; i=j \\
0 & \text{if $i$ and $j$ are not connected with an edge} \\
-1/2 & \text{if $i$ and $j$ are connected with an edge}.
\end{cases}
\]
Define the reflection $\sigma_i$ on $E$ by $\sigma_i(x) = x - 2B(x,x_i)x_i$.
For each $i$, $\sigma_i$ preserves the bilinear form $B$: $B(x,y) = B(\sigma_i(x),\sigma_i(y))$.
The assignment $w_i \in W$ to $\sigma_i$ gives an injective group homomorphism $\sigma:W \to GL(E)$
and all elements of $\sigma(W)$ preserve the bilinear form.

This representation of the Coxeter group is called the \emph{standard geometric representation}.

The Coxeter group $W$ also acts on the dual space $E^*$
via the induced action: if $f\in E^*$ is a functional on $E$, and $w \in W$,
then $(wf)$ is the functional defined by $(wf)(x) = f(w^{-1}(x))$.
The \emph{positive chamber} $\mathcal{C}$ is defined as the set of $f \in E^*$
such that $f(x_i) > 0$ for each $i = 0,\ldots, s-1$.
$\mathcal{C}$ is an open simplicial cone in $E^*$,
and no two elements of $\mathcal{C}$ can be in the same orbit of the Weyl group:
if $w \in W$ and $\mathcal{C} \cap w(\mathcal{C}) \neq \emptyset$ then $w=1$.
A reference for this material is \cite{Bo}, Chapter V, Section 4.

This `positive chamber' $\mathcal{C}$ has the set $\bar{\mathcal{C}}$ as its closure,
defined by the set of $f \in E^*$ such that $f(x_i) \geq 0$ for each $i$.
This set may be partitioned as follows:
for each subset $X$ of the basis vectors $\{x_i\}$,
we define $\mathcal{C}_X$ as the set of $f \in E^*$ such that
$f(x_i) = 0$ for $x_i \in X$, and $f(x_i) > 0$ for $x_i \notin X$.
We have $\mathcal{C}_\emptyset = \mathcal{C}$,
and $\bar{\mathcal{C}}$ is the disjoint union of all $\mathcal{C}_X$
as $X$ ranges over all subsets of the set of generators.

For such a subset $X$, define $W_X$ to be the subgroup of the Weyl group $W$
generated by the elements of $X$.
Proposition 5 of \cite{Bo}, Chapter V, Section 4, No. 6 gives the following:

\begin{proposition}\label{CXprop}
Suppose $X$ and $X'$ are subsets of the basis vectors $\{x_i\}$,
and $w$ is in the Weyl group $W$.
If $w(\mathcal{C}_X) \cap \mathcal{C}_{X'} \neq \emptyset$,
then $X=X'$,  $w \in W_X$, and $w(\mathcal{C}_X) = \mathcal{C}_{X}$.
In particular the stabilizer of any $f \in \mathcal{C}_X$ is $W_X$.
\end{proposition}

The union of all of the translates $w(\bar{\mathcal{C}})$ as $w$ ranges over the Weyl group
is the \emph{Tits cone} in $E^*$.
It is a convex cone in $E^*$,
and the non-negative cone $\bar{\mathcal{C}}$ is a fundamental domain for the action of $W$ on the Tits cone.
The general theory gives:

\begin{theorem}\label{FiniteWeyl}
The following are equivalent:
\begin{itemize}
\item The Weyl group is finite.
\item The Tits cone is the entire dual space $E^*$.
\item The bilinear form $B$ is positive definite.
\end{itemize}
\end{theorem}

The following is a useful observation.

\begin{proposition}\label{TitsAlgorithm}
Suppose that $g$ is a functional in the Tits cone,
so that some translate $w \in W$ sends $g$ into the non-negative cone $\bar{\mathcal{C}}$.
One can construct the Weyl group element $w$ that sends $g$ into $\bar{\mathcal{C}}$
with the following algorithm.
We may assume $g \notin \bar{\mathcal{C}}$, else take $w=1$.
Set $g_1 = g$, and set $k=1$.
\begin{itemize}
\item[(a)] If $g_k \in \bar{\mathcal{C}}$, stop.
\item[(b)] If $g_k \notin \bar{\mathcal{C}}$,
then choose any generator $x_{i_k}$ for $W$ such that $g_k(x_{i_k}) < 0$.
Set $g_{k+1} = x_{i_k}(g_k)$.
\item[(c)] Increment $k$ and go to step (a).
\end{itemize}
This process stops in a finite number $K$ of steps,
and $w = x_{i_{K}}x_{i_{K-1}}\cdots x_{i_2}x_{i_1}$ sends $g$ into $\bar{\mathcal{C}}$.
\end{proposition}

\medskip\noindent{\bf Proof:}
The only thing to show here is that the algorithm stops,
and we do this by induction on the length of the Weyl group element $w$ that takes $g$ into $\bar{\mathcal{C}}$.
If that length is $0$, then $g=g_1$ is already in the non-negative cone.
If not, then the algorithm allows us to choose a generator $x_{i_1}$ such that $g_1(x_{i_1}) < 0$,
and we consider $g_2 = x_{i_1}g_1$;
for convenience set $u = x_{i_1}$, so that $g_2 = u g_1$.
The Weyl group element $w u$ sends $g_2$ into the non-negative cone:
$(w u)(g_2) = (w u)(ug_1) = w g \in \bar{\mathcal{C}}$.

We claim that the length $\ell(w u)$ is less than the length $\ell(w)$.
This is equivalent to $\ell(u w^{-1}) < \ell(w^{-1})$,
which by \cite{Humphreys}, Lemma, page 120,
is equivalent to $w^{-1}(\mathcal{C}) \subset A_u'$,
where $A_u' = \{f \in E^* \;|\; f(u) < 0\}$.

Consider the original functional $g$; by assumption it lies in $A_u'$,
and since $w(g) \in \bar{\mathcal{C}}$,
we have $g \in w^{-1}(\bar{\mathcal{C}})$.
Hence $g$ is in the intersection $w^{-1}(\bar{\mathcal{C}}) \cap A_u'$,
so that intersection is not empty.

However any translate of $\mathcal{C}$ is fully contained in either $A_u'$
or $A_u = \{f \in E^* \;|\; f(u) > 0\}$.
If $w^{-1}(\mathcal{C}) \subset A_u$, 
then we must have $w^{-1}(\bar{\mathcal{C}}) \subset \bar{A_u}$.
However $\bar{A_u} \cap A_u' = \emptyset$;
this would then force
$w^{-1}(\bar{\mathcal{C}}) \cap A_u' = \emptyset$.
This is a contradiction, since $g$ is an element there.
Hence we conclude that $w^{-1}(\mathcal{C}) \subset A_u'$,
which then allows us to conclude that the length $\ell(wu) < \ell(w)$.

We finish then by induction on the length.

\rightline{QED}

\begin{remark}\label{Worbitcheck}
The above give us a practical method in many cases for determining
whether two elements $\tau$ and $\tau'$ of the Tits cone
are in the same orbit for the Weyl group.
First, use the algorithm of Proposition \ref{TitsAlgorithm}
to move $\tau$ and $\tau'$ into the non-negative chamber $\bar{\calC}$,
obtaining elements $c$ and $c'$.
Determine the subsets $X$ and $X'$
for which $c \in \calC_X$ and $c' \in \calC_{X'}$.
If $\tau$ and $\tau'$ are in the same orbit,
then so are $c$ and $c'$,
and therefore by Proposition \ref{CXprop} we must have $X=X'$,
and $c' = w(c)$ for some element in $W_X$.

Therefore if $X \neq X'$,
or $X=X'$ but no element of $W_X$ sends $c$ to $c'$,
we conclude that $\tau$ and $\tau'$ are not in the same orbit.
However if $X=X'$ and such an element exists,
then $\tau$ and $\tau'$ are  in the same orbit of course.
\end{remark}

\subsection{The Weyl Group Action on $A^1(Y)$}
The $\phi$-invariant symmetric curve class $F = (r+1)h - \sum_i e_i  \in A^{r-1}(Y)$
gives a functional on the divisor classes $A^1(Y)$
via the intersection pairing;
the kernel of this functional is a codimension one subspace $V^1 \subset A^1$
on which $\phi$ and the symmetric group acts irreducibly.
A divisor class $dH-\sum_i m_i E_i$ is in $V^1$ if and only if $(r+1)d = \sum_i m_i$.
(Using the notation of Proposition \ref{CremonaAY}(a),
we see that this is indeed preserved by $\phi_I$ for any index set $I$,
since the quantity $t_1$ is added to the degree and to $r+1$ of the multiplicities.)

We have $\dim(V^1) = s$ and a basis for $V^1$ is given by the divisor classes
$v_i = \frac{1}{r+1}H-E_i$ for $i = 1,\ldots,s$.
The symmetric group acts on these basis elements by permuting the indices.
The action of $\phi$ is given by
\[
\phi(v_i) = \begin{cases}
v_i - \frac{2}{r+1}\sum_{j=1}^{r+1}v_j \;\;\text{if}\;\; 1 \leq i \leq r+1 \\
v_i + \frac{r-1}{r+1}\sum_{j=1}^{r+1}v_j \;\;\text{if}\;\; r+2 \leq i \leq s
\end{cases}
\]
An alternative basis which is more tuned to the Coxeter theory is
$\{X_i\}_{0\leq i \leq s-1}$
where $X_0 = \sum_{i=1}^{r+1} v_i = H - \sum_{i=1}^{r+1} E_i$
and $X_i = v_{i+1}-v_i = E_i - E_{i+1}$ for $1 \leq i \leq s-1$.

The quadratic invariant $q_1$ on $A^1$
introduced in (\ref{q1})
gives rise to a (symmetric, $\phi$-invariant) bilinear form $B_1$ on $A^1$
defined by setting
\[
B_1(x,y) = \frac{-1}{4}(q_1(x+y)-q_1(x)-q_1(y))
\]
which takes on the following values:
\[
B_1(H,H) = \frac{1-r}{2}; \;\;\;
B_1(E_i,E_i) = 1/2; \;\;\;
B_1(H,E_i) = 0; \;\;\;
B_1(E_i,E_j) = 0\;\;\text{for}\;\; i\neq j
\]
and hence has signature $(s,1)$ on $A^1$.
We note that
since the canonical class $K = (r+1)H-(r-1)\sum_i E_i$,
and we have
$B_1(K,dH-\sum_i m_iE_i) = \frac{1-r}{2}( (r+1)d-\sum_i m_i )$,
we see that
\[
B_1(x,K) = \frac{1-r}{2}(x\cdot F)
\]
and hence $V^1$ is exactly the orthogonal space to the canonical class $K$,
with respect to the $B_1$ pairing.
We have that
\[
B_1(K,K) = \frac{1-r}{2}( (r+1)^2 - s(r-1) )
\]
and hence if $(r+1)^2 > s(r-1)$ we have $B_1(K,K) < 0$,
which implies that the restriction of $B_1$ to the subspace $V^1$
has signature $(s,0)$, i.e., $B_1$ is positive definite on $V^1$.

It has the following values on the $X_i$ basis vectors for $V^1$:
\begin{align*}
B_1(X_i,X_i) &= 1 \;\;\text{for all}\;\; i \\
B_1(X_0,X_i) &= \begin{cases}
0 \;\;\text{if}\;\; i \geq 1, i \neq r+1 \\
-1/2 \;\;\text{if}\;\; i = r+1\\
\end{cases} \\
B_1(X_i,X_j) &= 0 \;\;\text{if}\;\; 1 \leq i, j \leq s-1, |i-j| > 1 \\
B_1(X_i,X_{i+1}) &= -1/2 \;\;\text{if}\;\; i \geq 1
\end{align*}

Let $W^1$ be the group of transformations of $V^1$
generated by the symmetric group on the $s$ indices
and the Cremona action $\phi$.
Denote by $\tau_i$ the transposition $(i,i+1)$ in the symmetric group;
since the $\tau_i$'s generate the symmetric group,
the $\tau_i$'s and $\phi$ form a set of generators for the group $W^1$.
Each of these generators have order two,
and it is standard that
$\tau_i$ and $\tau_j$ commute if $|i-j| > 1$, else $\tau_i\tau_j$ has order $3$.
In addition, an easy computation shows that
$\tau_i$ and $\phi$ commute unless $i = r+1$, and in that case $\tau_{r+1}\phi$ has order $3$.
Hence $W^1$ is a homomorphic image of the Coxeter group
whose Coxeter graph is the $T_{2, r+1, s-r+1}$ graph introduced above:
there is a surjective group homomorphism $\alpha:W \to W^1$.
This homomorphism sends $\sigma_0$ to $\phi$ and $\sigma_i$ to $\tau_i$ for $i \geq 1$.

Let $V^1_{\bbR} = V^1 \otimes \bbR$ be the real vector space with basis $X_i$.
Define a linear map $\beta: E \to V^1_\bbR$ by sending $x_i$ to $X_i$ for each $i = 0,1,\ldots,s-1$
and extending linearly.

What we have developed here is that the action of the Weyl group $W^1$ on $V^1_\bbR$
is isomorphic to the action of the Coxeter group $W$ on the vector space $E$:

\begin{proposition}\label{actionW1V1}
The group homomorphism $\alpha$ is an isomorphism.
The linear map $\beta$ is an isomorphism that preserves the bilinear forms:
\[
B_1(\beta(x),\beta(y)) = B(x,y) \;\;\text{for all}\;\; x, y \in E.
\]
The diagram
\[
\begin{matrix}
W \times E & \longrightarrow & E \\
(\alpha,\beta) \downarrow & & \beta \downarrow \\
W^1 \times V^1_\bbR & \longrightarrow & V^1_\bbR \\
\end{matrix}
\]
is commutative, where the rows are the actions given above.
\end{proposition}

\medskip\noindent{\bf Proof:}
One first checks that the diagram commutes,
and for this is suffices to show that it commutes on ordered pairs
$(\sigma_i,x_j) \in W \times E$.
This ordered pair is sent to $\sigma_i(x_j) = x_j - 2B(x_j,x_i)x_i$ in $E$,
and is then sent to $X_j - 2B(x_j,x_i)X_i$ in $V^1_\bbR$.
Going the other way, it is sent to
$(\phi,X_j)$ (if $i=0$) or $(\tau_i,X_j)$ (if $i \geq 1)$ in $W^1\times V^1_\bbR$;
then on to $\phi(X_j)$ or $\tau_i(X_j)$ depending on $i$.
Hence we must show that
\[
X_j - 2B(x_j,x_i)X_i = 
\begin{cases}
\phi(X_j) \;\;\text{if}\;\; i = 0\\
\tau_i(X_j) \;\;\text{if}\;\; i \geq 1.
\end{cases}
\]

First assume that $i \geq 1$ and $j \geq 1$.

If $|i-j|>1$, then $B(x_j,x_i) = 0$ and so the clockwise image is $X_j$.
The counterclockwise image is $\tau_i(X_j) = \tau_i(E_j-E_{j+1}) = E_j-E_{j+1} = X_j$ also.

If $|i-j|=1$, then $B(x_j,x_i) = -1/2$, so the clockwise image is
$X_j + X_i$.  This is $E_j - E_{j+1} + E_i - E_{i+1}$
which is  $E_{i-1}-E_{i+1}$ if $j=i-1$ or $E_i-E_{i+2}$ if $j = i+1$.
The counterclockwise image is $\tau_i(E_j - E_{j+1})$
which is $E_j-E_{i+1}=E_{i-1}-E_{i+1}$ if $j=i-1$ and is $E_i-E_{j+1} = E_i-E_{i+2}$ if $j=i+1$ as required.

Finally if $i=j$ then $B(x_j,x_i)=1$ so the clockwise image is $X_j-2X_i = -X_i$.
The counterclockwise image is $\tau_i(X_i) = \tau_i(E_i-E_{i+1}) = E_{i+1}-E_i = - X_i$ also.

Next assume $i=0$ and $j \geq 1$.
In this case we must show that $X_j - 2B(x_j,x_0)X_0 = \phi(X_j)$.
If $j \neq r+1$, then $B(x_j,x_0)=0$, so we must show that $\phi$ fixes $X_j$, which is clear.

For $j=r+1$, we have $B(x_j,x_0)=-1/2$, so we must show that
$\phi(X_{r+1}) = X_{r+1} + X_0$.
Using the notation and formulas of Proposition \ref{CremonaAY}(a),
we have that $X_{r+1} = E_{r+1}-E_{r+2}$ so that
$d=0$, $m_{r+1}=-1$, $m_{r+2}=1$, and all other $m_i=0$; hence $t_1 = 1$
so that the image under the Cremona has degree $1$, $m_i = 1$ for $1 \leq i \leq r$, $m_{r+1}=0$,
$m_{r+2}=1$, and $m_i=0$ for $i \geq r+3$.
This is $X_0+X_{r+1}$ as required.

Next assume $i \geq 1$ and $j=0$.
In this case we must show that $X_0 - 2B(x_0,x_i)X_i = \tau_i(X_0)$.

If $i \neq r+1$ then $B(x_0,x_i)=0$ so we must observe that $\tau_i$ fixes $X_0$, which is clear.

If $i = r+1$ then $\tau_{r+1}(X_0) = X_0 + X_{r+1}$ as required.

The final check is for $i=j=0$, and we must show that $\phi(X_0) = - X_0$.
Again using the notation of Proposition \ref{CremonaAY}(a), 
we have $t_1=-2$, and since the degree and multiplicities for $X_0$ are all $1$ up to index $r+1$,
adding $-2$ changes them to $-1$, switching the sign.

This finishes the computation that the diagram commutes.

It is obvious that $\beta$ is an isomorphism, since it is defined by identifying the two bases.

The preservation of the bilinear forms is immediate from the computation of $B$ and $B_1$ given above on the basis vectors.

To finish we must show that $\alpha$ is a group isomorphism, and it suffices to show it is $1$-$1$.
Suppose that $w \in W$ is in the kernel of $\alpha$.
Fix any vector $e \in E$;
since the diagram commutes, we have $\beta(w(e)) = (\alpha(w))(\beta(e)) = \beta(e)$ in $V^1_\bbR$
since $\alpha(w) = 1$ in $W^1$.
Since $\beta$ is an isomorphism, this implies that $w(e) = e$.
Since $e$ was arbitrary, this implies $w=1$ in $W$.

\rightline{QED}

This finishes the proof that the group generated by the symmetric group and the Cremona transformation,
acting on $V^1$,
is isomorphic to the standard geometric representation of the Coxeter group with graph $T_{2,r+1,s-r-1}$.

Therefore we can apply Theorem \ref{FiniteWeyl} and conclude the following:

\begin{theorem}
If $(r+1)^2 > s(r-1)$ then the Weyl group acting on $V^1$ (and hence on $A^1$) is finite,
and the Tits cone in the dual space $(V^1)^*$ is the whole space.
\end{theorem}

\medskip\noindent{\bf Proof:}
We have seen that this inequality with $r$ and $s$ implies that the bilinear form $B_1$ is positive definite.
The previous theorem then allows us to apply Theorem \ref{FiniteWeyl} to conclude.
\rightline{QED}
We note that
\begin{equation}\label{finiteWeylrs}
(r+1)^2 > s(r-1) \iff r=2, s \leq 8; r=3, s \leq 7; r=4, s \leq 8; r\geq 5, s \leq r+3.
\end{equation}

\subsection{The dual space to $V^1_\bbR$ as a quotient of $A^{r-1}$}

The intersection pairing in the Chow ring allows us to naturally identify the curve classes $A^{r-1}$
as the dual space to the divisor classes $A^1$.
The subspace $V^1\subset A^1$ defined as the orthogonal space to the curve class
$F = (r+1)h - \sum_i e_i$
therefore has, as its natural dual space,
the quotient $V^{r-1} = A^{r-1}/\langle F \rangle$;
we mod out the span of the vector $F$
which acts via intersection as identically zero on $V^1$.

Let us be explicit about the dual action on $V^{r-1}$;
we will restrict ourselves to computing the action by the generators of the Weyl group $W$.
For the $x_i$, $i \geq 1$, these act on $V^1$ by the transposition $(i,i+1)$ on the classes of the exceptional divisors $E_i$;  clearly these act then on $V^{r-1}$ by that same transposition on the $e_i$ classes ($\mod F$) in $V^{r-1}$.

For the generator $x_0$, it acts on $V^{1}$ by the Cremona transformation on divisors;
hence if we fix a class $c \in V^{r-1} (\mod F)$ we have that $(x_0(c))(D) = (c \cdot \phi^{-1}(D)) = (c \cdot \phi(D))$
for any divisor class $D \in V^1$.
Since the Cremona transformation preserves the intersection form on the Chow ring, we have
$(c \cdot \phi(D)) = (\phi(c) \cdot D)$,
so that $x_0(c)$ acts as intersection with $\phi(c)$,
and hence is equal to $\phi(c)$ by the perfect pairing.

Therefore the positive chamber $\mathcal{C} \subset V^{r-1}_\bbR$ is the set of curve classes $c (\mod F)$
such that $c \cdot X_i > 0$ for each generator $X_i \in V^1$.
Write $c = dh - \sum_i m_i e_i (\mod F)$.
Then we have:
\begin{align*}
(c \cdot X_0) &= (c \cdot H-\sum_{i=1}^{r+1}E_i) = d-\sum_{i=1}^{r+1}m_i \\
i \geq 1: (c \cdot X_i) &= (c \cdot E_i-E_{i+1}) = m_{i} - m_{i+1} \\
\end{align*}
(We note that these conditions are all invariant under the addition of multiples of the curve class $F$.)
Hence the positive chamber is defined by
\[
\mathcal{C} = \{dh-\sum_i m_i e_i (\mod F) \;|\; d > \sum_{i=1}^{r+1} m_i \;\;\text{and}\;\; m_i > m_{i+1} \forall i = 1,\ldots,s-1\}.
\]
This is the condition that the multiplicities strictly decrease, and that the class is strictly Cremona reduced.
The closure $\bar{\mathcal{C}}$ relaxes these conditions to $d \geq \sum_{i=1}^{r+1}m_i$ and $m_i\geq m_{i+1}$.
Note that the $\mod F$ requirement then eliminates the condition that all multiplicities are positive:
we do not have the condition that $m_s\geq 0$.

For each subset $X \subset \{X_i\}$ we also have the chamber faces $\mathcal{C}_X$ as above,
defined with the equalities instead of strict inequalities here.

It is useful to be explicit about the algorithm presented in Proposition \ref{TitsAlgorithm}.
We re-formulate it in our case below.

\begin{proposition}\label{cremonareduction}
Let $c \in A^{r-1}(\mod F)$, and suppose it is not in the non-negative chamber $\bar{\mathcal{C}}$.
Then we may apply the following sequence of transformations to $c$:
\begin{itemize}
\item[(a)] Reorder the multiplicities in descending order.
\item[(b)] If $c$ is now Cremona reduced, stop.
\item[(c)] If $c$ is not Cremona reduced,
apply the Cremona transformation $\phi$ (to the first, hence largest, $r+1$ multiplicities).
\item[(d)] Go to step (a).
\end{itemize}
If this sequence stops, then $c$ is in the Tits cone
and it stops with a Cremona-reduced class that is Weyl-equivalent to $c (\mod F)$.
If this sequence does not stop, then $c$ is not in the Tits cone,
and no Weyl group translate of $c$ can be Cremona reduced.
\end{proposition}

\begin{remark}\label{rem:monotonic}
Since applying Cremona transformations when the class is not Cremona reduced strictly reduces the degree, Proposition \ref{cremonareduction} implies that if a class is in the Tits cone, then the sequence of Cremona transformations that brings the class into the non-negative chamber (i.e., makes it Cremona reduced) monotonically reduces the degree.
\end{remark}

One practical consequence of Proposition \ref{CXprop} for our purposes is the following.

\begin{lemma}\label{twoCremonareducedclasses}
Suppose that $c$ and $c'$ are two Cremona reduced classes
with multiplicities in non-increasing order.
Assume that they are Cremona-equivalent:
there is an element $w$ in the Weyl group such that $w(c)=c'$.
Then $c \equiv c' \mod F$.
\end{lemma}

\begin{proof}
Suppose that $c \mod F \in \calC_X$ for a subset $X$ of generators of the Weyl group $W$.
Proposition \ref{CXprop} implies that $c' \mod F \in \calC_X$ also,
and $w \in W_X$.
Since $W_X$ fixed (pointwise) $\calC_X$,
we must have $c' \mod F = c \mod F$ as claimed.
\end{proof}

\subsection{The direct action on the space of curve classes}
In the treatment above, we see that the Coxeter group acts on the divisor classes $A^1(Y^r_s)$, with the action restricted to the subspace $V^1$ being isomorphic to the standard geometric representation;
with this perspective, the dual space is then the quotient of the curve classes.

It turns out that we can take the opposite approach:
the corresponding subspace of $A^{r-1}$ is also isomorphic to the
standard geometric representation of the Coxeter group,
and the corresponding quotient of $A^1$ is then the dual.

It is more convenient to tensor with $\bbR$ at the outset.
Define the vector space isomorphism
$\Psi: A^1(Y^r_s)\otimes\bbR \to A^{r-1}(Y^r_s)\otimes\bbR$
by
$\Psi(H)=h$ and $\Psi(E_i) = e_i/(r-1)$ for each $i$.

Note then that
\begin{align*}
q_{r-1}(\Psi(dH-\sum_i m_i E_i)
&= q_{r-1}(dh-\sum_i (m_i/(r-1)) e_i \\
&= d^2 - (r-1)\sum_i (m_i/(r-1))^2 \\
&= \frac{1}{r-1} ((r-1)d^2 - \sum_i m_i^2) \\
&= \frac{1}{r-1} q_1(dH-\sum_i m_i E_i)
\end{align*}
and hence the isomorphism $\Psi$ preserves the quadratic forms 
(and therefore the bilinear forms) up to scalar.

It is clear that $\Psi$ commutes with the action of the symmetric group.
We claim that it also commutes with the Cremona transformations.
To check this, it is enough to check it for $\phi$.

We have
$\phi(\Psi(dH-\sum_i m_i E_i))
= \phi(dh - \sum_i (m_i/(r-1))e_i)
= d'h - \sum_i m_i' e_i$
where if we define $t_{r-1} = d-\sum_{i\leq r+1} (m_i/(r-1))$, then
$d' = d+(r-1)t_{r-1}$,
$m_i' = m_i/(r-1) + t_{r-1}$ for $i \leq r+1$, and
$m_i'= m_i/(r-1)$ for $i > r+1$.

On the other hand,
$\Psi(\phi(dH-\sum_i m_i E_i))
= \Psi(d''H - \sum_i m_i'' E_i)
= d''h - \sum_i m_i''/(r-1) e_i$
where if we define $t_1 = (r-1)d-\sum_{i \leq r+1} m_i$, then
$d'' = d+t_1$,
$m_i''=m_i+t_1$ for $i \leq r+1$, and
$m_i'' = m_i$ for $i > r+1$.

We now see that $d'=d'' = rd-\sum_{i \leq r+1} m_i$, so the coefficients of $h$ agree.
If $j \leq r+1$, then the coefficient of $e_j$ in both cases is
$d-\sum_{i \leq r+1, i \neq j} m_i/(r-1)$.
Finally if $j > r+1$, the coefficient of $e_j$ in both cases is
$m_j/(r-1)$.

We conclude that $\Psi(\phi(C)) = \phi(\Psi(C))$ for all $C$.
We also note that $\Psi(K) = F$.

We can then reproduce the entire analysis and conclude the following:
\begin{proposition}
Let $V^{r-1} \subset A^{r-1}\otimes\bbR$
be the subspace orthogonal to $K$ under the intersection pairing $(-\cdot -)$.
Then the action of the Weyl group on $V^{r-1}$ is isomorphic
to the standard geometric representation of the Coxeter group for $T_{2,r+1,s-r1}$;
the dual space is isomorphic to the quotient space $A^1\otimes\bbR/<K>$.
\end{proposition}

We immediately obtain the analogue of Proposition \ref{cremonareduction} also,
which we leave to the reader to formulate.


\section{Applications}\label{applications}
There are two main applications of the Coxeter group/Weyl group circle of ideas that we want to focus on, dealing with the Weyl group orbit of a general line, a line through one point, and a line through two points.

It is elementary to see that such lines through $1-i$ points 
($i = 1,0,-1$)
are smooth rational curves in $Y^r_s$
with normal bundle equal to $\calO(i)^{\oplus(r-1)}$.
In Section \ref{rationalcurves(i)curves} we referred to such a smooth rational curve an \emph{$(i)$-curve} in $Y^r_s$;
If $c$ is the class of such a line,
then $\langle c, F\rangle = 2+i(r-1)$;
such a class is a \emph{numerical $(i)$-class} in $A^{r-1}$.

An \emph{$(i)$-Weyl line} is a curve in $Y^r_s$
that is in the orbit of a line through $1-i$ points
under the action of the Weyl group.
If $c$ is the class of such a curve,
then $\langle c, F\rangle = 2+i(r-1)$
and $\langle c, c\rangle = 1+(i-1)(r-1)$.
A \emph{numerical $(i)$-Weyl class}
is a class $c$ with these two values for the invariants.

Finally recall that a curve class $dh-\sum_i m_i$
with $d \geq 1$, $m_i \geq m_{i+1}$ for each $i$
and $m_s \geq 0$ is
\emph{Cremona reduced}
if $d \geq \sum_{i=1}^{r+1} m_i$;
this ensures that the Cremona transformations $\phi_I$
do not reduce the degree, for any set of $r+1$ indices $I$.

In \cite{DM2} the authors investigated the relationship between these concepts,
primarily in the case when $\langle F,F\rangle > 0$;
this is the case when $Y^r_s$ is a \emph{Mori Dream Space},
and the Weyl group is finite.
In this article we will focus primarily on the other cases,
although we will be able to draw some conclusions in the Mori Dream Space cases as well.

The approach is as follows.
Take the line class through two points $h-e_1-e_2$ (in the $i = -1$ case),
the line class through one point $h-e_1$ (in the $i=0$ case),
or the general line class $h$ (if $i=1$).

Consider the class $\mod F$ in $A^{r-1}/\langle F \rangle$.
The first step is to determine whether the class $\mod F$ is in the Tits cone,
by using Proposition \ref{cremonareduction}.

If the class is in the Tits cone, determine the class $c^+\;\mod F$ in the nonnegative chamber $\bar\calC$ that is in its orbit.
Lift the analysis to $A^{r-1}$ and classify all Cremona reduced classes there that are in the orbit of the line class.

If the class is not in the Tits cone, we conclude that no Cremona reduced class can be in the orbit of that degree one class in $A^{r-1}$.

\subsection{The Weyl group orbit of the line through two points: planar and finite Weyl group case}
Our primary example will be the line $L_{12}$ through the first two points,
whose class is $h-e_1-e_2$; $d=m_1=m_2=1$, and all other multiplicities are zero.
We note that $h-e_1-e_2 \cdot X_0 = -1$ and hence this class $(\mod F)$ is not in the positive chamber; it is not Cremona reduced.
In the finite Weyl group case, every class is in the Tits cone,
and so the algorithm of Proposition \ref{cremonareduction}
will terminate in a Cremona reduced class $R_{r,s} \;\mod F$ in the orbit of $h-e_1-e_2$.
It turns out that this happens for $r=2$ and all $s \geq 3$ as well.
We present the results for the various cases below.
We also present the subset $X$ for which $R_{r,x} \in \calC_X$,
where $\calC$ is positive chamber for the Weyl group action.
\[
\begin{matrix}
r, s & R_{r,s}=(d;\underline{m})\; \mod F & \text{ subset }X \\ \hline
r=2, s =3: & (0; 0,0, -1) & \{X_1\} \\
r=2, s \geq 4: & (0; 0^{s-1}, -1) & \{X_{i\neq s-1}\} \\
r \geq 3, s=r+1: & (2-r;0,0,(-1)^{r-1}) & \{X_{i \neq 0,2}\}\\
r \geq 3, s=r+2: & (2-r; 0,0,0,(-1)^{r-1}) & \{X_{i\neq 3}\}\\
r=2k \geq 2, s=r+3: & ((2k+1)(1-k); (1-k)^{r+2}, -k) & \{X_{i \neq s-1}\}\\
r=2k+1 \geq 3, s= r+3: & (-2k^2-2k+1; (-k)^{r+3}) & \{X_{i \neq 0}\}\\
r=3, s=7: & (-3;0,(-1)^6) & \{X_{i\neq 1}\}\\
r=4, s=8: & (-14;-2,(-3)^7) & \{X_{i\neq 1}\}\\
\end{matrix}
\]

We note that in all cases, the subgroup $W_X$ fixes the Cremona reduced element $R_{r,s}$ as expected.

Now assume $r=2$ or $r \geq 3$ and \eqref{finiteWeylrs} holds,
i.e., the Weyl group is finite.

Suppose that $c$ is a numerical $(-1)$-Weyl curve class in $A^{r-1}$;
i.e., $\langle c, F\rangle = 3-r$ and $\langle c, c \rangle = 3-2r$.
To detect whether $c$ is the class of a $(-1)$-Weyl line,
we apply the algorithm of Proposition \ref{cremonareduction}
and arrive at a class $R$.  We then consider $R \;\mod F\in \bar\calC$.
By Remark \ref{Worbitcheck},
if $c$ is the class of a $(-1)$-Weyl line,
we must have $R = R_{r,s} \;\mod F$,
since the subgroup $W_X$ fixes $R_{r,s}$.
If not, we conclude that $c$ is not the class of a $(-1)$-Weyl line.
If so, then we conclude that $R = R_{r,s}+kF$ for some $k$.
We note that both $R$ and $R_{r,s}$
have the same linear and quadratic invariants;
they are numerical $(-1)$-Weyl lines too.

Now it is elementary that, for these values of the invariants,
if $\alpha$ is any curve class,
such that $\langle \alpha+kF,F\rangle = \langle \alpha,F\rangle$
and $\langle \alpha+kF,\alpha+kF \rangle = \langle \alpha,\alpha\rangle$,
then we must have $k=0$.
(Indeed, the linear invariant equation suffices unless $r=2$ and $s=9$,
since this is the only case where $\langle F,F\rangle = 0$.)

We conclude that $R = R_{r,s}$,
and so $c$ reduces to $R_{r,s}$ by applying the reduction algorithm,
and hence reduces to something with non-positive degree and non-positive multiplicities.

Finally we remark that $c$ actually must reduce to the original line class $h-e_1-e_2$
on the way to $R_{r,s}$.
To see this, one makes a case-by-case analysis of each $R_{r,s}$,
asking what classes could reduce to it by applying the Cremona reduction algorithm.
Let us show how this analysis goes in one case; we will leave the others
(which are entirely parallel) to the reader.

For example, consider the $r=3$, $s=7$ case;
here $R_{3,7} = (-3;0,(-1)^6)$ using the vector notation;
we will re-brand this class as $R_{-3}$.
Since the standard Cremona transformation is an involution,
one obtains all classes that could Cremona to $R_{-3}$
by applying the standard Cremona to $R_{-3}$.
There are two choices here (up to symmetries) for choosing the $4$ points to apply the standard Cremona transform: either we include the $0$-point or not.
If we include the $0$-point, the Cremona transform fixes $R_{3,7}$.
If we do not, the Cremona transform leads to $R_{-1}=(-1;0^5,(-1)^2)$.
Hence any class that reduces to $R_{3,7}$ must 'pass through' $R_{-1}$ on the way.

Now apply the same argument to $R_{-1}$;
if the four chosen points include zero of the $(-1)$-points,
the transformation goes back to $R_{-3}$.
If the four chosen points include one of the $(-1)$-points
then it fixes $R_{-1}$
If the four chosen points include two of the $(-1)$-points,
the Cremona brings it to $R_1 = (1;1^2,0^5)$,
which is the original line class.
We then conclude that any class that reduces to $R_{3,7}$ must pass through this original line class along the way.

This happens in each case in the table above; as noted, we leave this elementary check to the reader.

Finally we note that the only class that directly reduces to $h-e_1-e_2$
is the class of the rational normal curve $(r+1;1^{r+3},0^{s-r-3})$;
hence if $c$ is not the line class already,
then it must actually reduce to the RNC class as well, on the way.

We have proved the following.

\begin{proposition}\label{prop:MDS-1case}
Suppose that $r=2$ and $s \geq 3$
or $r,s$ satisfy \eqref{finiteWeylrs}.
Let $c$ be a class in $A^{r-1}$
with positive degree and non-negative multiplicities,
such that $\langle c, F \rangle = 3-r$.
Then:
\begin{itemize}
\item[(a)] If $c$ is Cremona reduced, it is not the class of a $(-1)$-Weyl line.
\item[(b)] If $c$ is not Cremona reduced, 
let $R$ be the result of applying the Cremona reduction algorithm.  
If $r\neq 2$ or $r =2$ and $s \neq 9$, 
then $c$ is the class of a $(-1)$-Weyl line
if and only if $R=R_{r,s}$,
which is equivalent to $c$ Cremona-reducing to $h-e_1-e_2$
via the reduction algorithm.
If $c \neq h-e_1-e_2$, then this happens
if and only if $c$ reduces to the class of the rational normal curve through $r+3$ points.
\item[(c)] If $r=2$ and $s=9$, and $\langle c, c \rangle =3-2r$,
then $c$ is the class of a $(-1)$-Weyl line
if and only if $R=R_{r,s}$,
which is equivalent to $c$ Cremona-reducing to $h-e_1-e_2$
via the reduction algorithm.
If $c \neq h-e_1-e_2$, then this happens
if and only if $c$ reduces to the class of the rational normal curve through $r+3$ points.
\end{itemize}
\end{proposition}

The content here is that if a class is such that both the linear and quadratic invariant
is the same as line through two points, it is either reducible by Cremona,
or it is Cremona reduced, and is in fact not a Weyl line.

The Proposition above indicates that we must apply the Cremona reduction algorithm to determine whether a class is the class of a $(-1)$-Weyl line.
Of course necessary conditions are that the class have the correct linear and quadratic invariants, and (in the $r,s$ cases considered above), not be Cremona reduced if the degree and multiplicities are non-negative.  These are not sufficient conditions however, as the following examples show.

\begin{example}\label{examplesd=7and13}
\begin{itemize}
\item[(a)]
We recall from \cite{DM2} that
the curve $G=7h-4e_1-\sum_{i=2}^{11} e_i$ is not a $(-1)$-Weyl line
even though it satisfies the linear and quadratic equalities
$\langle G, F\rangle = 3-r$
and $\langle G, G\rangle = 3-2r$.
(It is Cremona reduced.)
We observe that $G$ fails the projection inequality \eqref{projectioninequalityiclass} 
in $\bbP^3$.

\item[(b)] We observe that 
even if a curve satisfies the linear and quadratic equalities
 and the projection inequality 
it is not enough to guarantee that the curve class is a $(-1)$ Weyl line.
Indeed, consider $J:=13h-\sum_{i=1}^5 4e_i -\sum_{i=6}^{11} e_i$ in $\bbP^3$;
it does satisfy the linear and quadratic invariants and the projection inequality \eqref{projectioninequalityiclass}.
However $J$ is not the class of a $(-1)$ Weyl line.
\end{itemize}
\end{example}


\subsection{Lines through two points: infinite Weyl group cases}

Here we want to prove that
$h-e_1-e_2$ is not in the Tits cone.
To see this, we will show that running the algorithm of Proposition \ref{cremonareduction} does not terminate.

\begin{lemma}
Suppose $r\geq 3$, $s \geq r+5$, and $c = (d;m_1,\ldots,m_s)$ is a class in $A^{r-1}$
with $m_i \geq m_{i+1}$ for each $i$, and satisfying
\[
\text{(i): }d < M_1, \;\; \text{(ii): }m_s \geq m_1+t, \;\;\text{and (iii): } 2d \leq M_1+M_{s-r},
\]
where $M_k = \sum_{i=0}^{r} m_{k+i}$ and $t=d-M_1$.
Then if $c'=(d';m_1',\ldots,m_s')$ is the class obtained from $c$ 
by applying $\phi$ and re-ordering the multiplicities in descending order,
then the degree and multiplicities of $c'$ also satisfy the above three inequalities.
\end{lemma}

\medskip\noindent{\bf Proof:}
The inequality (i) says that the degree strictly decreases,
and (ii) implies that, in descending order, the numerical characters $(d'\underline{m'})$ of $c'$ are:
\begin{gather*}
d'=d+(r-1)t=rd+(r-1)M_1;\\
m_1'=m_{r+2},\ldots,m_{s-r-1}'=m_s,m_{s-r}'=m_1+t, m_{s-r+1}'=m_2+t,\ldots,m_s'=m_{r+1}+t.
\end{gather*}
Since the multiplicities are in descending order, (iii) implies that
\begin{equation}\label{2dM1A}
2d \leq M_1 + A
\end{equation}
where $A$ is any sum of $r+1$ distinct multiplicities,
since $M_{s-r}\leq A$ in that case.
By adding to this any multiple $j$ of (i), we also have that
\begin{equation}\label{jdM1A}
(j+2)d < (j+1)M_1+A
\end{equation}
for any $j \geq 1$, if $A$ is any sum of $r+1$ distinct multiplicities.

We note that 
\[
M_1' = \begin{cases} M_{r+2} & \text{ if } s \geq 2r+2 \\
\sum_{i=r+2}^s m_i + m_1+\ldots+m_k  + kt & \text{ if } k=2r+2-s \geq 1
\end{cases}
\]
In both cases we have that $M_1'=A+kt$ where
$A$ is a sum of $r+1$ distinct multiplicities and $0 \leq k \leq r-3$.
(We define $A$ and $k$ by this,
and we observe that we use $s \geq r+5$ to see that $k \leq r-3$.)

The persistence of (i) requires us to show that $d'<M_1'$,
which is equivalent to $d+(r-1)t < A+kt$, or $(r-k)d< (r-k-1)M_1+A$.
This is (\ref{jdM1A}) with $j=r-k-2$; since $k \leq r-3$, we have $j\geq 1$ as needed.

To show that (ii) persists, we must show that $m_1'+t' \leq m_s'$.
Since $t'=d'-M_1'$, this is equivalent to
$m_{r+2}+(rd-(r-1)M_1)-(A+kt) \leq m_{r+1}+t$, which can be rearranged as
$(r-k-1)d \leq (r-k-2)M_1 +A+ m_{r+1}-m_{r+2}$.
Since $m_{r+1} \geq m_{r+2}$,
it suffices to show that $(r-k-1)d \leq (r-k-2)M_1 +A$;
since $r-k-1 \geq 2$, this follows from (\ref{2dM1A}) (if $k=r-3$) or (\ref{jdM1A}) (if $k < r-3$).

Finally to show that (iii) persists,
we must show that $2d' \leq M_1'+M_{s-r}'$.
We note that $M_{s-r}' = M_1+(r+1)t$, so this is equivalent to
$2(rd-(r-1)M_1) \leq (A+kt) + (M_1+(r+1)t)$.
This can be rearranged as
$(r-k-1)d \leq (r-k-2)M_1 + A$
which is the same inequality we just showed above.

\rightline{QED}

If $s=r+4$ the above proof must be modified as follows.

\begin{lemma}
Suppose $r\geq 5$, $s = r+4$, and $c = (d;m_1,\ldots,m_s)$ is a class in $A^{r-1}$
with $m_i \geq m_{i+1}$ for each $i$, and satisfying
\[
\text{(i): }d < M_1, \;\; \text{(ii): }m_1+t \leq m_{r+4}, \;\;\text{(iii): } d \leq N_1 \;\;\text{and (iv): }
2d \leq M_1+M_4,
\]
where $M_k = \sum_{i=0}^{r} m_{k+i}$,
$N_1 = m_1+m_2+m_3+\sum_{i=7}^{r+4 }m_i$,
and $t=d-M_1$.
Then if $c'=(d';m_1',\ldots,m_s')$ is the class obtained from $c$ 
by applying $\phi$ and re-ordering the multiplicities in descending order,
then the degree and multiplicities of $c'$ also satisfy the above three inequalities.
\end{lemma}

\medskip\noindent{\bf Proof:}
The inequalities (i) and (ii) imply that when we apply $\phi$ and re-order the multiplicities,
the result has the degree strictly decreasing,
with $d' = d+(r-1)t = rd-(r-1)M_1$;
and the multiplicities of $c'$ are
\[
m_1'=m_{r+2},m_2'=m_{r+3},m_{3}'=m_{r+4},m_{4}'=m_1+t, m_{5}'=m_2+t,\ldots,m_{r+4}'=m_{r+1}+t.
\]
Define $N_2 = \sum_{i=1}^{r-2} m_i + m_{r+2}+m_{r+3}+m_{r+4}$,
and observe that $N_2 \geq N_1$ since $r\geq 5$.
We note that
$M_1' = N_2+(r-2)t$,
$N_1' = M_4+(r-2)t$,
and $M_4' = M_1+(r+1)t$.

To prove that (i) persists, we must show that $d' < M_1'$,
which is equivalent to $2d < M_1+N_2$; this follows from (i) and (iii).

To prove that (ii) persists, we must show that $m_1'+t' \leq m_{r+4}'$,
which is equivalent to $d \leq N_2+(m_{r+1}-m_{r+2}$;
this is implied by (iii) and that the multiplicities are in descending order.

To prove that (iii) persists, we must show that $d'\leq N_1'$,
which is equivalent to (iv).

To prove that (iv) persists, we must show that $2d' \leq M_1'+M_4'$,
which is equivalent to $d \leq N_2$, and is implied therefore by (iii).

\rightline{QED}

Since the hypotheses of these two Lemmas 
are satisfied by the class $h-e_1-e_2$,
they imply that the class of a line through two points
is not in the Tits cone,
for the range of parameters that enjoy an infinite Weyl group.
Hence no Cremona reduced class can be in the orbit of this class, and we have proved the analogue of Proposition \ref{prop:MDS-1case} for these cases.
This covers all possible parameters then, and we have the following.

\begin{theorem}\label{thm:noCR-1}
For all parameters $r,s$, with $r \geq 2$ and $s \geq r+1$,
there are no Cremona-reduced numerical $(-1)$-Weyl classes
$c = (d;\underline{m})$
with all $d$, $m_i$ non-negative
which are $(-1)$-Weyl lines, i.e., Cremona-equivalent to $h-e_1-e_2$.
\end{theorem}

Note that we are not saying that there are no curve classes with all $d, m_i$ non-negative that are classes of $(-1)$-Weyl lines; there are infinitely many.
The class $(d;\underline{m}) = (1;1^2)$ is one of course;
so is the class of the rational normal curve
$(r+1;1^{r+3})$ in $\bbP^r$ through $r+3$ general points.
If we allow $s$ to increase, we can find such classes with arbitrarily high degree
using Cremona transformations.

This also gives us the following.

\begin{proposition}\label{monotonic}
Suppose that $c$ is a $(-1)$-Weyl line class with non-negative degree and multiplicities,
in the orbit of $h-e_1-e_2$.
Then $c$ is not Cremona reduced.
If the degree is larger than $r$,
the algorithm of re-ordering the multiplicities in non-increasing order
and applying the standard Cremona transformation $\phi$ to the $r+1$ largest multiplicities
will strictly reduce the degree, and can be applied iteratively
until one arrives at the class of the rational normal curve through $r+3$ points.
That class reduces to $h-e_1-e_2$ with one more application of $\phi$.
\end{proposition}

This is simply the content of Lemma \ref{numericallemma}(d),
after observing that we now know that all classes are non-Cremona-reduced.
The importance of this is that, applying Cremona transformations as above,
one reduces a $(-1)$-Weyl line class 'monotonically' to the class of a rational normal curve.

\subsection{The case of a line through one point}
The same principles and arguments can be applied to other interesting curve classes in $A^{r-1}$.
The main difference is that this class is in the Tits cone, so the analysis must be adjusted appropriately; it is more elementary in fact.

We first take up the case of the line through one point,  whose class is $h-e_1$.
This class represents a rational curve in $\bbP^r$ with trivial normal bundle.

\begin{theorem}\label{thm:0weyllines}
Fix $r\geq 3$.
Then the only Cremona-reduced classes $c = (d;\underline{m})$
(with all $d$, $m_i$ non-negative)
which are Cremona-equivalent to $h-e_1$
are the classes $h-e_i$ for $1 \leq i \leq s$.
Any $(0)$-Weyl line with positive degree and non-negative multiplicities
may be brought to $h-e_1$ via a series of Cremona transformations,
each of which monotonically decreases the degree,
followed by a single permutation.
\end{theorem}

\begin{proof}
Take $c$ and permute the multiplicities to be in non-increasing order.
Then both $h-e_1 \mod F$ and $c\mod F$
are in the closure $\bar\calC$ of the positive chamber, and so we may apply
Lemma \ref{twoCremonareducedclasses}
and conclude (since they are Cremona equivalent)
that they are equal $\mod F$.
Hence we may write $c = aF + (h-e_1)$ for some integers $a \geq 0$.

Since $\langle c,F\rangle = \langle h-e_1,F\rangle = 2$,
this forces $a\langle F,F\rangle = 0$.
If $\langle F, F \rangle \neq 0$, we conclude that $a=0$.

Suppose that $\langle F, F \rangle = 0$.
Then $2-r = \langle c,c\rangle = \langle h-e_1,h-e_1\rangle + 2a\langle h-e_1,F\rangle
= 2-r +4a$
and we conclude $a=0$ in this case as well.

Hence $c = h-e_1$ after permuting the multiplicities,
and so $c = h-e_j$ for some $j$ as claimed.

The final statement follows from applying the Cremona reduction algorithm.

\end{proof}


\subsection{The case of a line through zero points}
A parallel (even easier) argument can be applied to the line class $h$ also:

\begin{theorem}\label{thm:1weyllines}
Fix $r\geq 3$.
Then the only Cremona-reduced classes $c = (d;\underline{m})$
(with all $d$, $m_i$ non-negative)
which are Cremona-equivalent to $h$
is the class $h$.
Any $(1)$-Weyl line with positive degree and non-negative multiplicities
may be brought to $h$ via a series of Cremona transformations,
each of which monotonically decreases the degree.

\end{theorem}

\begin{proof}
Suppose that $c$ is such a class; permute the multiplicities again to be in non-increasing order.
Then as above we conclude that $c = h \;\mod F$,
and we may write $c = aF + h$ for some $a \geq 0$.

Using the bilinear form pairing with $F$ gives that
$r+1 = \langle c, F\rangle = \langle h,F\rangle + a\langle F,F\rangle
= r+1+ a\langle F,F\rangle$,
so we conclude $a=0$ if $\langle F,F\rangle \neq 0$.

If $\langle F,F \rangle = 0$, then
$1 = \langle c,c\rangle = \langle h+aF,h+aF\rangle = 1 + 2a\langle h,F\rangle = 1+2(r+1)a$,
and we conclude that $a=0$ in this case as well.
\end{proof}

\subsection{The Projection Inequality for $(i)$-Weyl lines}

Fix $i \in \{-1,0,1\}$, and suppose that $C$ is an $(i)$-Weyl line,
Cremona equivalent to $h-e_1-e_2$, $h-e_1$, or $h$.

\begin{corollary}\label{projinequalityalliWeyllines}
Fix $i \in \{-1,0,1\}$ and suppose that $c$ is an $(i)$-Weyl line class with non-negative degree and multiplicities.
Assume $c \neq h-e_i-e_j$, a degree one $(-1)$-Weyl line.
Then $c$ satisfies the projection inequality \eqref{projectioninequalityicurves},
(equivalently \eqref{bilinearprojinequality})
i.e.
$d + m_1 \leq \sum_{j=2}^s m_j +i$
or
$\langle c, h-e_1\rangle \geq 1$.
\end{corollary}

\begin{proof}
Start with $i=-1$.
This is an induction on the degree $d$,
starting with $d=r$ by Lemma \ref{numericallemma}(b).
For $d=r$, $c$ must be the rational normal curve, which does satisfy (\ref{projectioninequalityicurves}).
The induction step now follows from Proposition \ref{monotonic},
using Proposition \ref{cremonaprojection} to see that the inequality is preserved.

For $i=0$ or $i = 1$, again we apply induction,
using Proposition \ref{cremonaprojection} for the induction step
and this time noting that the statement is true for $h-e_j$ and $h$, as starting points for the induction.  The required monotonicity is provided by Theorems \ref{thm:0weyllines} and \ref{thm:1weyllines}.
\end{proof}

\section{Criteria for $(-1)$-Weyl lines}
We have seen above that if $c$ is the class of a $(-1)$-Weyl line in $A^{r-1}(Y^r_s)$,
then we must have the specific values for the linear and quadratic invariants
and the projection inequality must hold (unless it is the class $h-e_i-e_j$).

We have also seen in Example \ref{examplesd=7and13}(b) 
that these numerical conditions are not sufficient.
The issue there was that the system satisfied the projection inequality, but after one Cremona transformation, it did not.  This leads us to consider a stronger version.

\subsection{The Strong Projection Inequality}
We first recall that the projection inequality for $(i)$-curves \eqref{projectioninequalityicurves}
can be re-written as in \eqref{bilinearprojinequality}:
\[
\langle c,h-e_j\rangle \geq 1 \;\text{ for every } j
\]
and this can be checked by applying it only to the largest multiplicity index $j$.

Given the simplicity of the planar case (see Lemma \ref{planarcase}),
we will assume that $r \geq 3$.

We have already seen, in Corollary \ref{projinequalityalliWeyllines},
that if a class $c$ is an $(i)$-Weyl line class,
then it satisfies the projection inequality
(unless $i=-1$ and the degree equals one).
We can extend this as follows.

\begin{lemma}\label{iWeyllineclassisnumerical}
Let $c\in A^{r-1}(Y^r_s)$ be an $(i)$-Weyl line class 
with positive degree and non-negative multiplicities.
Let $w$ be an element of the Weyl group
such that $\deg(w^{-1}(c))>1$.
Then for every $j$,
\[
\langle c, w(h-e_j) \rangle \geq 1.
\]
\end{lemma}

\begin{proof}
By hypothesis $c$ is in the orbit of $h-e_1-e_2$, $h-e_1$, or $h$,
and therefore $w^{-1}(c)$ is as well.
Therefore since $w^{-1}(c)$ has degree bigger than one,
Corollary \ref{projinequalityalliWeyllines} applies and we have
$\langle w^{-1}(c), h-e_j \rangle \geq 1$.
This implies the result, because the bilinear form is invariant under the Weyl group,
so that $\langle c, w(h-e_j) \rangle = \langle w^{-1}(c), h-e_j \rangle$.
\end{proof}

\begin{definition}\label{strong}
Let $c\in A^{r-1}(Y^r_{s})$ be a curve class.
We say that $c$ satisfies the \emph{strong projection inequality} if
\[
\langle c, w(h-e_j)\rangle \geq 1
\]
for any element $w$ of the Weyl group for which $\deg(w^{-1}(c))>1$.
\end{definition}

Hence any $(i)$-Weyl line class satisfies the strong projection inequality,
by Lemma \ref{iWeyllineclassisnumerical}.

\begin{definition}
A curve class $c \in A^{r-1}$ is a \emph{numerical $(i)$-Weyl line class} if
\begin{itemize}
\item[(a)] $\langle c, F \rangle = 2+i(r-1)$;
\item[(b)] $\langle c, c \rangle = (r-1)(i-1)+1$;
\item[(c)] for any $w$ in the Weyl group with $\deg(w^{-1}(c))>1$ then
$\langle c, w(h-e_j)\rangle\geq 1,$
\end{itemize}
i.e., $c$ has the same linear and quadratic invariants as an $(i)$-Weyl line,
and satisfies the strong projection inequality.
\end{definition}	

\begin{corollary}\label{iWeylisnumerical}
Every $(i)$-Weyl line class in $\bbP^r$
(i.e., a class in the orbit of $h-e_1-e_2$, $h-e_1$, or $h$)
is a numerical $(i)$-Weyl line class.
\end{corollary}

This follows immediately from Corollary \ref{iWeyllineinvariants}
Lemma \ref{iWeyllineclassisnumerical}.

We will leave the following as an open question in $\bbP^r$. 
\begin{question}\label{quest}
Is the converse true?
Are the two notions of $(i)$-Weyl line class
(i.e., a class in the orbit of $h-e_1-e_2$, $h-e_1$, or $h$) 
and numerical $(i)$-Weyl line class
(i.e., satisfying the above three numerical conditions)
equivalent in $\bbP^r$?
\end{question}

We devote the next section to the study of curves in $\bbP^3$.
We generalize Max Noether's inequality for planar curves
to curves in $\bbP^3$
and provide a positive answer to this question 
for classes on general blowups of $\mathbb{P}^3$.

\subsection{M. Noether's inequality for curves in $\mathbb{P}^3$}\label{noether section}
The aim of this section is to answer affirmatively Question \eqref{quest} in $\bbP^3$
by proving a Noether type inequality for curves,
which was originally proved by M. Noether for planar curves.
The same proof was later generalized to divisors in $\bbP^r$ \cite{DO, DP}.
We show below that the proof also applies for curves in $\bbP^3$.

We remark that for numerical $(i)$-classes in $\bbP^3$
the projection inequality reads $m_j\leq (d-1)/2$
by \eqref{projectioninequalityiclass}.

\begin{theorem}[M. Noether's inequality for curves in $\bbP^3$] \label{noether}
Fix $i \in \{-1,0,1\}$.
Let $c = dh-\sum_i m_i \in A^{2}(Y^3_{s})$ be a curve class,
with degree $d$ greater than one.
Assume that 
\begin{enumerate}
  	\item 	$\langle c, F \rangle = 2i+2$. (This gives $2d=\sum_{i=1}^s m_i + i + 1$.)
  	\item	$\langle c, c \rangle \leq 2i-1$. (This gives $d^2 \leq 2\sum_{i=1}^s m_i^2 +2i-1$.)
  	\item	$m_k \leq (d-1)/2 \text{ for any }  k\in \{1, \ldots, s\}.$
 \end{enumerate}
Then $m_1+m_2+m_3+m_4 > d$.
\end{theorem}

\begin{proof}
If $s \leq 4$ then the first condition implies 
$m_1+\ldots+m_s=2d-i-1 \geq 2d-2$.
If $d \geq 3$ then $2d-2 \geq d+1 > d$, which suffices.
If $d=2$ then $2d-i-1 \leq d$ only if $i=1$;
and in that case the second condition says that $3 \leq 2\sum m_i^2$.
However the third condition forces all multiplicities to be zero, giving a contradiction.

We consider now $s\geq 5$. 
 We put the multiplicities in non-increasing order: 
$m_k \geq m_p$ for $k\leq p$.
For any $j\in \{1, \ldots, s\}$,  we define
\[
q_j:=\frac{m_j^2+\ldots+ m_s^2}{m_j+\ldots +m_s}
\]
and observe that $m_j \geq q_j$  for any $j\in \{1, \ldots, s\}$,
which follows from the decreasing order of the multiplicities,
since  $\sum_{k=j}^{s} m_j  m_k\geq \sum_{k=j}^s m_k^2$.

Notice that
\begin{equation}\label{eq3}
\begin{split}
 q_1: &= \frac{m_1^2+\ldots+ m_s^2}{m_1+\ldots +m_s} \\
 &\geq \frac{(d^2+(1-2i))/2}{2d-i-1} \\
 &> \frac{d}{4}
\end{split}
\end{equation}
The last inequality is equivalent to $d(1+i) > (4i-2)$ which holds for $-1 \leq i\leq 1$ and $d>1$.


Define $r_k = m_k-q_k$ for all $k$; note that $r_k \geq 0$.
Define $p_k = \frac{m_{k-1}}{m_k+\cdots + m_s}$ for $k >1$.
Note that
\begin{equation}\label{eq2}
\begin{split}
q_{k-1}-r_{k-1}p_k
	&=q_{k-1} -(m_{k-1}-q_{k-1})p_k \\
	&= -m_{k-1}p_k + q_{k-1}(1 + p_k) \\
	&=  -m_{k-1}p_k + q_{k-1}\frac{m_{k-1}+\cdots _+ m_s}{m_k+\cdots + m_s} \\
	&= -m_{k-1}p_k + \frac{m_{k-1}^2 +\cdots + m_s^2}{m_k+\cdots + m_s} \\
	&= \frac{-m_{k-1}^2+m_{k-1}^2 +\cdots + m_s^2}{m_k+\cdots + m_s} \\
&= q_k.
\end{split}
\end{equation}
This implies that
\begin{equation}\label{eq4}
\begin{split}
q_{k-1}+r_{k-1}+ N q _k &= q_{k-1}+r_{k-1} + N(q_{k-1}-r_{k-1}p_k) \\
&= (N+1)q_{k-1} + r_{k-1}(1-Np_k)
\end{split}
\end{equation}
for any $N$.

Several of these quantities $1-Np_k$ are non-negative.
To see this first note that $2d-4m_1-i-1 \geq 0$,
since this is equivalent to $m_1 \leq \frac{d-1}{2} + \frac{1-i}{4}$,
which follows from assumption 3, given that $i \in \{-1,0,1\}$.
Since $m_1$ is the largest multiplicity,
this implies that
\begin{equation}\label{fourmults}
2d-k-i-1 \geq 0
\end{equation}
for $k$ equal to the sum of any four multiplicities.

Now we see the following:
\begin{equation}\label{eq1}
\begin{split}
1-3p_2 &= \frac{-3m_1+m_2+\cdots + m_s}{m_2+\cdots + m_s}
	= \frac{2d-4m_1-i-1}{m_2+\cdots + m_s} \geq 0,\\
1-2p_3 &= \frac{-2m_2+m_3+\cdots + m_s}{m_3+\cdots + m_s}
	= \frac{2d-m_1-3m_2-i-1}{m_3+\cdots + m_s} \geq 0,\;\;\text{and}\\
1-p_4 &= \frac{-m_3+m_4+\cdots + m_s}{m_4+\cdots + m_s}
	= \frac{2d-m_1-m_2-2m_3-i-1}{m_4+\cdots + m_s} \geq 0,
%
\end{split}
\end{equation}
all of which use \eqref{fourmults} applied to the numerator.

Now compute, using \eqref{eq3}, \eqref{eq2}, \eqref{eq4}, and \eqref{eq1}:
\[
\begin{split}
m_1+m_2+m_3+m_4 & \geq (q_1+r_1) + (q_2+r_2) + (q_3+r_3) + q_4 \\
&=(q_1+r_1) + (q_2 + r_2) + 2 q_3 +	r_3 \cdot (1-p_4)\\
&\geq (q_1+r_1) + (q_2 + r_2) + 2q_3 \\
&= (q_1+r_1) + 3q_2 + r_2(1-2p_3) \\
&\geq  (q_1+r_1) + 3q_2 \\
& = 4q_1 + r_1(1-3p_2) \\
&\geq 4q_1 > d
\end{split}
\]
\end{proof}

Hence any class satisfying the hypotheses of Theorem \ref{noether}
is not Cremona reduced, and we may apply a standard Cremona transformation
at the four points with largest multiplicities to strictly reduce the degree.

\begin{remark}
An $(i)$-Weyl line in $Y^3_s$ satisfies the hypotheses of Theorem \ref{noether}
since it has has the correct linear and quadratic invariants
and satisfies the third condition by Lemma \ref{iWeyllineclassisnumerical}.
\end{remark}

\subsection{A numerical criterion for Weyl curves in $\mathbb{P}^3$}
We are now in a position to prove the main result of this section.


\begin{theorem}\label{equivalence2}
Let $C$ be an irreducible curve in $Y^3_s$ with class $c = dh-\sum_i m_i e_i \in A^2(Y^3_s)$ 
such that $d>1$ (and all multiplicities non-negative).
Then the following are equivalent:
\begin{itemize}	
\item[(a)] 
	\begin{enumerate}
		\item 	$\langle c, F \rangle=2i+2$. 
		\item	$\langle c, c \rangle= 2i-1$.
		\item	$\langle c, w(h-e_k)\rangle\geq 1$
			for any $w$ in the Weyl group with $\deg(w^{-1}(c))>1$.
	\end{enumerate}
\item[(b)] $c$ is a $(i)$-Weyl line class 
and represents an irreducible $(i)$-curve $C$ in $Y^3_s$ 
which is Cremona-equivalent to a line through $1-i$ points.
\end{itemize}
\end{theorem}

\begin{proof}
We have seen that (b) implies (a)
via Corollary \ref{iWeylisnumerical}.
We claim now that (a) implies (b); we will prove this by induction on the degree $d$.
	
{\bf 1. Base Step.} We assumed that $C$ is a curve of degree $d \geq 2$.
The first observation is that there exist no curve of degree $d=2$
satisfying condition (a) (1), (2), (3);
if $i=\pm 1$ then (a)(2) requires $d$ to be odd,
and if $i=0$ there is no integer solution to (1) and (3).
If $d=3$, the only integer solution to the assumptions
is the twisted cubic with multiplicity one at $5-i$ points:
$c=3h-\sum_{k=1}^{5-i} e_k$.
	
{\bf 2. Induction Step.} Assume now that $d \geq 4$,
and we assume the statement holds for any curve $D$ 
with $3\leq \deg(D) \leq d-1$.
We will prove the result for a curve $C$ of degree $d$,
with non-negative multiplicities, which we may arrange in non-increasing order.
We assume then that such a $C$ satisfies (a)(1,2,3).

The assumption (a)(3) (applied with the identity Weyl group element and $k=1$) 
implies that $m_1\leq (d-1)/2$
and therefore (a) implies that hypothesis of Theorem \ref{noether} are satisfied.
Therefore the class $c$ is not Cremona reduced.
Since $d=\deg(C) \geq 3$ we also have that $C$ is non-degenerate, by Lemma \ref{numericallemma}(c).
Hence the standard Cremona transformation $\phi$ will reduce the degree of $C$.
If we let $D:=\phi(C)$, then $D$ is also an irreducible curve:
$C$ cannot be contracted by $\phi$ since it is non-degenerate.
We also have $\deg(D) < \deg(C)$.

Furthermore,
$\langle D, F \rangle = \langle \phi(D), \phi(F) \rangle = \langle C,F\rangle = 2i+2 $
and
$\langle D, D \rangle= \langle C, C \rangle= 2i-1$.
Therefore (1) and (2) of (a) hold for $D$;
we claim now that (3) is also true for $D$.

To see this, fix any $w$ in the Weyl group with $\deg(w^{-1}(D))>1$.
We then have
$$
\langle D, w(h-e_k)\rangle
=\langle \phi(D), (\phi\circ w)(h-e_k)\rangle
= \langle C, \eta(h-e_k)\rangle
$$
where $\eta=\phi \circ w$.
We claim now that $\deg(\eta^{-1}(C))>1$.
Indeed, $\phi$ is an involution, so
\begin{equation}
\begin{split}
  	\deg \eta ^{-1}(C) &=\deg( (\phi \circ w)^{-1}(C))\\
  	&=\deg  (w^{-1} \circ \phi^{-1})(C)\\
  	&=\deg w^{-1} (\phi (C))\\
&=\deg w^{-1}(D)\\
&>1.
\end{split}
\end{equation}
This proves the claim, and
therefore by the hypothesis (a)(3) on $C$, we have $\langle C, \eta(h-e_k)\rangle \geq 1$.
Since
$$
\langle D, w(h-e_k)\rangle = \langle \phi(C), \phi(\eta(h-e_k))\rangle
= \langle C, \eta(h-e_k)\rangle \geq 1,
$$
the hypothesis (a)(3) also holds for $D$ as well.
Hence by induction, we conclude that $D$ is a $(i)$-Weyl line,
i.e., in the orbit of a line through $1-i$ points;
therefore so is $C$.

\end{proof}
 

\end{document}